\documentclass[reqno,12pt]{amsart}

\usepackage{amsmath, epsfig, cite}
\usepackage{amssymb}
\usepackage{amsfonts}
\usepackage{latexsym}
\usepackage{amsthm}

\newtheorem{theorem}{Theorem}[section]

\newtheorem{corollary}[theorem]{Corollary}
\newtheorem{conjecture}[theorem]{Conjecture}
\newtheorem{lemma}[theorem]{Lemma}
\newtheorem{problem}[theorem]{Problem}

\theoremstyle{remark}

\numberwithin{equation}{section}

\newcommand{\ta}{\theta}

\newcommand{\ba}{\beta}

\newcommand{\da}{\delta}

\newcommand{\rd}{\,\mathrm d}
\newcommand{\C}{\mathbb C}
\newcommand{\N}{\mathbb N}

\newcommand{\lcm}{\operatorname{lcm}}

\allowdisplaybreaks

\author{Victor J.\ W.\ Guo}
\address{School of Mathematics and Statistics, Huaiyin Normal University,
Huai'an 223300, Jiangsu, People's Republic of China}
\email{jwguo@hytc.edu.cn}
\thanks{The first author was partially supported by the National Natural
Science Foundation of China (grant 11771175).
The second author was partially supported by Austrian Science Fund grant P32305.}

\author{Michael J.\ Schlosser}
\address{Fakult\"at f\"ur Mathematik, Universit\"at Wien,
Oskar-Morgenstern-Platz~1, A-1090 Vienna, Austria}
\email{michael.schlosser@univie.ac.at}

\title[$q$-supercongruences from transformations]{Some
$q$-supercongruences from transformation formulas for basic hypergeometric
series}



\begin{document}

\begin{abstract}
Several new $q$-supercongruences are obtained using transformation
formulas for basic hypergeometric series, together with various
techniques such as suitably combining terms, and creative
microscoping, a method recently developed by the first author in
collaboration with Zudilin. More concretely, the results in
this paper include $q$-analogues of supercongruences (referring to
$p$-adic identities remaining valid for some higher power of $p$)
established by Long, by Long and Ramakrishna, and several other
$q$-supercongruences. The six basic hypergeometric transformation
formulas which are made use of are Watson's transformation, a
quadratic transformation of Rahman, a cubic transformation of Gasper
and Rahman, a quartic transformation of Gasper and Rahman, a double
series transformation of Ismail, Rahman and Suslov, and a new
transformation formula for a nonterminating very-well-poised
${}_{12}\phi_{11}$ series. Also, the nonterminating $q$-Dixon
summation formula is used. A special case of the new
${}_{12}\phi_{11}$ transformation formula is further utilized to
obtain a generalization of Rogers' linearization formula for the
continuous $q$-ultraspherical polynomials.
\end{abstract}

\maketitle

\section{Introduction}\label{secintro}
Ramanujan, in his second letter to Hardy on February 27, 1913,
mentioned the following identity
\begin{equation}
\sum_{k=0}^{\infty} (-1)^{k}(4k+1)\frac{(\frac{1}{2})_k^5}{k!^5}=
\frac{2}{\Gamma(\frac{3}{4})^4},  \label{eq:ram5}
\end{equation}
where $\Gamma(x)$ is the Gamma function and where $(a)_k=a(a+1)\cdots(a+k-1)$
is the Pochhammer symbol. A $p$-adic analogue of \eqref{eq:ram5}
was conjectured by Van Hamme~\cite[Eq.~(A.2)]{Hamme} as follows:
\begin{equation}
\sum_{k=0}^{(p-1)/2} (-1)^{k}(4k+1)\frac{(\frac{1}{2})_k^5}{k!^5}
\equiv
\begin{cases} -\displaystyle\frac{p}{\Gamma_p(\frac{3}{4})^4} \pmod{p^3},
&\text{if $p\equiv 1\pmod 4$,}\\
 0\pmod{p^3}, &\text{if $p\equiv 3\pmod 4$.}
\end{cases} \label{eq:a2}
\end{equation}
Here and throughout the paper, $p$ always denotes an odd prime and
$\Gamma_p(x)$ is the $p$-adic Gamma function.
The congruence \eqref{eq:a2} was later proved by McCarthy and Osburn~\cite{MO}
through a combination of ordinary and Gaussian hypergeometric series.
Recently, the congruence \eqref{eq:a2} for $p\equiv 3\pmod{4}$ and $p>3$
was further generalized by Liu \cite{Liu} to the modulus $p^4$ case.

It is well known that some truncated hypergeometric series are closely
related to Calabi-Yau threefolds over finite fields and are further relevant
to the coefficients of modular forms.
For example, using the fact that the Calabi-Yau threefold in question is
modular, which was proved by Ahlgren and Ono~\cite{AO},
Kilbourn~\cite{Kilbourn}
succeeded in proving Van Hamme's (M.2) supercongruence:
\begin{equation}
\sum_{k=0}^{(p-1)/2} \frac{(\frac{1}{2})_k^4}{k!^4} \equiv  a_p \pmod{p^{3}},
\label{eq:kil}
\end{equation}
where $a_p$ is the $p$-th coefficient of a weight $4$ modular form
\begin{equation*}
\eta(2z)^4\eta(4z)^4:=q\prod_{n=1}^\infty (1-q^{2n})^4  (1-q^{4n})^4,
\quad q=e^{2\pi iz}. 
\end{equation*}
Applying Whipple's $_7F_6$ transformation formula, Long~\cite{Long} proved that
\begin{equation}
\sum_{k=0}^{(p-1)/2} (4k+1) \frac{(\frac{1}{2})_k^6}{k!^6}
\equiv p \sum_{k=0}^{(p-1)/2} \frac{(\frac{1}{2})_k^4}{k!^4}\pmod{p^{4}}
\quad\text{for $p>3,$}  \label{eq:long-1}
\end{equation}
which in view of the supercongruence \eqref{eq:kil} can be written as
\begin{equation*}
\sum_{k=0}^{(p-1)/2} (4k+1) \frac{(\frac{1}{2})_k^6}{k!^6}
\equiv p\,a_p \pmod{p^{4}}\quad\text{for $p>3$.}
\end{equation*}

The main aim of this paper is to give $q$-analogues of some known
supercongruences, including a \textit{partial}\/ $q$-analogue of Long's
supercongruence \eqref{eq:long-1} (partial in the sense that the
modulo $p^4$ condition is replaced by the weaker condition modulo $p^3$).
We provide such a result in Theorem~\ref{thm:first} in the form of two
transformations of truncated basic hypergeometric series. In addition, several
other $q$-supercongruences are given. These results are proved by
special instances of transformation formulas for basic hypergeometric series.
(See Theorem~\ref{newtf} in the Appendix for a new basic
hypergeometric transformation formula which we make use of.)

Throughout we assume $q$ to be fixed with $0<|q|<1$. We refer to $q$ as
the ``base''. For $a,k\in\C$, the $q$-shifted factorial is defined by
\begin{equation}\label{qpoch}
(a;q)_k:={}\frac{(a;q)_\infty}{(aq^k;q)_\infty},\qquad
\text{where}\qquad (a;q)_\infty={}\prod_{j\geqslant 0}(1-aq^j).
\end{equation}
For brevity, we frequently use the shorthand notation
\begin{equation*}
(a_1,\dots,a_m;q)_k=(a_1;q)_k\dots(a_m;q)_k,\qquad\qquad
k\in\C\cup\infty.
\end{equation*}
Moreover, the {\it $q$-binomial coefficients}
$\begin{bmatrix}\begin{smallmatrix}x\\k\end{smallmatrix}\end{bmatrix}$
are defined by
\begin{equation*}
\begin{bmatrix}x\\k\end{bmatrix}
=\begin{bmatrix}x\\k\end{bmatrix}_q
=\begin{cases}\displaystyle\frac{(q^{x-k+1};q)_k}{(q;q)_k},
&\text{if $k\geqslant 0$,} \\[5pt]
0, &\text{otherwise.}
\end{cases}
\end{equation*}
It is easy to see that
\begin{equation}
(-1)^k q^{k^2}\begin{bmatrix}-\frac{1}{2}\\k\end{bmatrix}_{q^2}
=\frac{(q;q^2)_k}{(q^2;q^2)_k}
=\frac{1}{(-q;q)_k^2}\begin{bmatrix}2k\\k\end{bmatrix}. \label{eq:q-bino-1/2}
\end{equation}

Following Gasper and Rahman~\cite{GR}, basic hypergeometric
${}_{r}\phi_s$ series with $r$ upper parameters
$a_1,\dots,a_r$, $s$ lower parameters $b_1,\dots,b_s$, base $q$ and
argument $z$ are defined by
\begin{equation*}
{}_{r}\phi_s\!\left[\begin{matrix}
a_1,a_2,\dots,a_r\\b_1,\dots,b_s
\end{matrix};q,z\right]:=\sum_{k=0}^\infty
\frac{(a_1,a_2,\dots,a_r;q)_k}{(q,b_1,\dots,b_s;q)_k}
\left[(-1)^k q^{\binom k2}\right]^{1+s-r}z^k,
\end{equation*}
where $q\neq 0$ when $r>s+1$.
Such a series terminates if one of the upper parameters, say,
$a_r$, is of the form $q^{-n}$, where $n$ is a nonnegative integer.
If the series does not terminate, then it converges for $|z|<1$.

In many of our proofs we will make use of Watson's $_8\phi_7$
transformation formula \cite[Appendix~(III.17)]{GR}:
\begin{align}
& _{8}\phi_{7}\!\left[\begin{array}{cccccccc}
a,& qa^{\frac{1}{2}},& -qa^{\frac{1}{2}}, & b,    & c,    & d,    & e,    & f    \\
  & a^{\frac{1}{2}}, & -a^{\frac{1}{2}},  & aq/b, & aq/c, & aq/d, & aq/e, & aq/f
\end{array};q,\, \frac{a^2q^2}{bcdef}
\right] \notag\\
&\quad =\frac{(aq, aq/de, aq/df, aq/ef;q)_\infty}
{(aq/d, aq/e, aq/f, aq/def;q)_\infty}
\,{}_{4}\phi_{3}\!\left[\begin{array}{c}
aq/bc,\ d,\ e,\ f \\
aq/b,\, aq/c,\, def/a
\end{array};q,\, q
\right],  \label{eq:8phi7}
\end{align}
which is valid whenever the $_8\phi_7$ series converges and the $_4\phi_3$
series terminates. In particular, we will also make use of the limiting case
$f=q^{-n}\to\infty$,  which we state for convenience:
\begin{align}
&\sum_{k=0}^{\infty}\frac{(-1)^k(1-aq^{2k})(a,b,c,d,e;q)_k}
{(1-a)(q,aq/b,aq/c,aq/d,aq/e;q)_k}
\left(\frac{a^2q^2}{bcde}\right)^k q^{\binom k2} \notag\\[5pt]
&\quad=\frac{(aq, aq/de;q)_\infty}{(aq/d, aq/e;q)_\infty}
\,{}_{3}\phi_{2}\!\left[\begin{array}{c}
aq/bc,\ d,\ e\\
aq/b,\, aq/c
\end{array};q,\, \frac{aq}{de}
\right].  \label{eq:8phi7-2}
\end{align}
Other transformations we make use of are a quadratic transformation formula
of Rahman, stated in \eqref{eq:1-aq3k}, a cubic transformation formula of
Gasper and Rahman, stated in \eqref{eq:1-acq4k}, a quartic transformation
formula by Gasper and Rahman, stated in \eqref{eq:quartic}, a double series
transformation by Ismail, Rahman and Suslov, stated in \eqref{eq:irs},
and a new transformation formula for a
nonterminating ${}_{12}\phi_{11}$ series into two multiples of
nonterminating ${}_4\phi_3$ series, given as Theorem~\ref{newtf}
in the Appendix.
We also make use of the $q$-Dixon summation, stated in \eqref{eq:q-Dixon}.

For further material on basic hypergeometric series and more generally,
to special functions, we refer to the text books by Gasper and
Rahman~\cite{GR}, and by Andrews, Askey and Roy~\cite{AAR},
respectively. In particular, in our computations we implicitly make heavy
use of elementary manipulations of $q$-shifted factorials (see
\cite[Appendix~I]{GR}).

Recall that the {\it $q$-integer} is defined as $[n]=[n]_q=1+q+\cdots+q^{n-1}$.
Moreover, the $n$-th {\it cyclotomic polynomial} $\Phi_n(q)$ is given by
\begin{equation*}
\Phi_n(q):=\prod_{\substack{1\leqslant k\leqslant n\\ \gcd(n,k)=1}}(q-\zeta^k),
\end{equation*}
where $\zeta$ is an $n$-th primitive root of unity.
It is clear that $\Phi_n(q)$ is a polynomial in $q$ with integer coefficients.
Further,
\begin{equation*}
\prod_{d\mid n,\, d>1}\Phi_d(q)=[n],
\end{equation*}
in particular, $\Phi_p(q)=[p]$ for prime $p$.

We say that two rational functions  $A(q)$ and $B(q)$ in $q$ are congruent
modulo a polynomial $P(q)$, denoted by $A(q)\equiv B(q)\pmod{P(q)}$,
if the numerator of the reduced form of $A(q)-B(q)$ is divisible by $P(q)$
in the polynomial ring $\mathbb{Z}[q]$.
We refer the reader to \cite{Andrews99,CP,Guo2018,Guo-a2,Guo-jmaa,Guo-rama,Guo-mod4,Guo4,GS1,GS2,GS3,GZ15,LPZ,LP,NP,SP,Tauraso1,Tauraso2,WY,Zudilin2}
for some interesting $q$-congruences.

\section{The main results}
The following is our $q$-analogue of \eqref{eq:long-1}, where the modulo $p^4$
condition is replaced by the weaker condition modulo $p^3$.

\begin{theorem}\label{thm:first}
Let $n$ be a positive odd integer. Then
\begin{subequations}
\begin{align}
\sum_{k=0}^{n-1}[4k+1]\frac{(q;q^2)_k^6}{(q^2;q^2)_k^6} q^k
&\equiv [n]q^{(1-n)/2} \sum_{k=0}^{(n-1)/2} \frac{(q;q^2)_k^4}{(q^2;q^2)_k^4} q^{2k}
\pmod{[n]\Phi_n(q)^2}, \label{eq:first-0}\\
\intertext{and}
\sum_{k=0}^{(n-1)/2}[4k+1]\frac{(q;q^2)_k^6}{(q^2;q^2)_k^6} q^k
&\equiv [n]q^{(1-n)/2} \sum_{k=0}^{(n-1)/2} \frac{(q;q^2)_k^4}{(q^2;q^2)_k^4} q^{2k}
\pmod{[n]\Phi_n(q)^2}.  \label{eq:first}
\end{align}
\end{subequations}
\end{theorem}
Noticing that the terms corresponding to $k$ in the upper half range
$(n-1)/2<k\leqslant n-1$ are congruent to $0$ modulo $\Phi_n(q)^3$ but not
modulo $[n]\Phi_n(q)^2$ in general, we conclude that \eqref{eq:first-0}
and \eqref{eq:first} are in fact different congruences. Of course, when $n=p$
is an odd prime and $q=1$, they are both equivalent to \eqref{eq:long-1}
modulo $p^3$.
The proof of Theorem~\ref{thm:first} is deferred to Section~\ref{sec:proofs}.

Van Hamme~\cite[Eq.~(H.2)]{Hamme} proved the following supercongruence:
\begin{equation}
\sum_{k=0}^{(p-1)/2} \frac{(\frac{1}{2})_k^3}{k!^3}
\equiv
\begin{cases} -\displaystyle \Gamma_p\bigg(\frac{1}{4}\bigg)^4  \pmod{p^2},
&\text{if $p\equiv 1\pmod 4$,}\\
 0\pmod{p^2}, &\text{if $p\equiv 3\pmod 4$.}
\end{cases} \label{eq:h2}
\end{equation}
The first author and Zeng~\cite[Cor.~1.2]{GZ15} gave a $q$-analogue
of \eqref{eq:h2} as follows:
\begin{align*}
&\sum_{k=0}^{(p-1)/2}\frac{(q;q^2)_k^2 (q^2;q^4)_k}{(q^2;q^2)_k^2 (q^4;q^4)_k}
q^{2k}\\[5pt]
&\quad{}\equiv
\begin{cases} \displaystyle 
\begin{bmatrix}(p-1)/2\\(p-1)/4\end{bmatrix}_{q^4}^2
\frac{q^{(p-1)/2}}{(-q^2;q^2)_{(p-1)/2}^2}  \pmod{[p]^2},
&\text{if $p\equiv 1\pmod 4$,}\\
 0\pmod{[p]^2}, &\text{if $p\equiv 3\pmod 4$.}
\end{cases}
\end{align*}
We do not know any $q$-analogue of \eqref{eq:a2}. However, we are able
to provide a $q$-analogue of a very closely related identity.
In particular, since $\Gamma_p(\frac{1}{4})^4 \Gamma_p(\frac{3}{4})^4=1$,
from \eqref{eq:a2} and \eqref{eq:h2} we deduce that
\begin{equation}
\sum_{k=0}^{(p-1)/2} (-1)^{k}(4k+1)\frac{(\frac{1}{2})_k^5}{k!^5}
\equiv p\sum_{k=0}^{(p-1)/2}  \frac{(\frac{1}{2})_k^3}{k!^3} \pmod{p^3},
\label{eq:a2-h2}
\end{equation}
which was already noticed by Mortenson~\cite{Mortenson}.
We are able to give the following complete $q$-analogue of
\eqref{eq:a2-h2}.
\begin{theorem}\label{thm:second}
Let $n$ be a positive odd integer. Then
\begin{subequations}
\begin{align}
\sum_{k=0}^{n-1}(-1)^k [4k+1]\frac{(q;q^2)_k^5}{(q^2;q^2)_k^5} q^{k^2+k}
&\equiv [n]q^{(1-n)/2}
\sum_{k=0}^{(n-1)/2} \frac{(q;q^2)_k^3}{(q^2;q^2)_k^3} q^{2k} \pmod{[n]\Phi_n(q)^2},
\label{eq:second-0} \\
\intertext{and}
\sum_{k=0}^{(n-1)/2}(-1)^k [4k+1]\frac{(q;q^2)_k^5}{(q^2;q^2)_k^5} q^{k^2+k}
&\equiv [n]q^{(1-n)/2}
\sum_{k=0}^{(n-1)/2} \frac{(q;q^2)_k^3}{(q^2;q^2)_k^3} q^{2k} \pmod{[n]\Phi_n(q)^2}.
\label{eq:second}
\end{align}
\end{subequations}
\end{theorem}

Note that, just like in Theorem~\ref{thm:first}, the two congruences
\eqref{eq:second-0} and \eqref{eq:second} are not equivalent.
The proof of Theorem~\ref{thm:second} is deferred to Section~\ref{sec:proofs}.

Long and Ramakrishna~\cite[Thm.~2]{LR} proved the following supercongruence:
\begin{equation}
\sum_{k=0}^{p-1} (6k+1) \frac{(\frac{1}{3})_k^6}{k!^6}
\equiv
\begin{cases} -p\displaystyle \Gamma_p\bigg(\frac{1}{3}\bigg)^9
\pmod{p^6}, &\text{if $p\equiv 1\pmod 6$,}\\[10pt]
 -\frac{p^4}{27}\displaystyle \Gamma_p\bigg(\frac{1}{3}\bigg)^9\pmod{p^6},
&\text{if $p\equiv 5\pmod 6$.}
\end{cases} \label{eq:d2}
\end{equation}
This result is stronger than Van Hamme's (D.2) supercongruence conjecture which
asserts a congruence modulo $p^4$ for $p\equiv 1\pmod 6$.
Long and Ramakrishna also pointed out that \eqref{eq:d2} does not hold
modulo $p^7$ in general.

We propose the following partial $q$-analogue of
Long and Rama\-krishna's supercongruence \eqref{eq:d2}.
\begin{theorem}\label{thm:third}
Let $n$ be a positive integer coprime with $3$. Then
\begin{equation}
\sum_{k=0}^{n-1} [6k+1]\frac{(q;q^3)_k^6}{(q^3;q^3)_k^6} q^{3k}
\equiv
\begin{cases} 0  \pmod{[n]}, &\text{if $n\equiv 1\pmod 3$,}\\[10pt]
 0 \pmod{[n]\Phi_n(q)}, &\text{if $n\equiv 2\pmod 3$.}
\end{cases}  \label{eq;3rd-noa}
\end{equation}
\end{theorem}

We also partially confirm the $a=1$ case of the second congruence in \cite[Conj.~5.2]{GuoZu}.
\begin{theorem}\label{thm:fourth}
Let $d$ and $n$ be positive integers with $d>2$ and $n\equiv -1\pmod d$. Then
\begin{equation}
\sum_{k=0}^{n-1}[2dk+1]\frac{(q;q^d)_k^4}{(q^d;q^d)_k^4}q^{(d-2)k}
\equiv 0\pmod{\Phi_n(q)^2}.  \label{eq:2dk+1}
\end{equation}
\end{theorem}
The proofs of Theorems~\ref{thm:third} and \ref{thm:fourth} are deferred to
Section~\ref{sec:proofs34}.

In Section~\ref{sec:proofs}, we shall prove Theorems~\ref{thm:first}
and \ref{thm:second} using the {\it creative microscoping} method
developed by the first author and Zudilin~\cite{GuoZu}. Roughly speaking,
to prove a $q$-supercongruence modulo $\Phi_n(q)^3$, we prove its
generalization with an extra parameter $a$
so that the corresponding congruence holds modulo $\Phi_n(q)(1-aq^n)(a-q^n)$.
Since the polynomials $\Phi_n(q)$, $1-aq^n$, and $a-q^n$ are relatively prime,
this generalized $q$-congruence can be established modulo these three
polynomials individually.
Finally, by taking the limit $a\to 1$, we obtain the original
$q$-supercongruence of interest.
We learned that this creative microscoping method has already caught
the interests of Guillera~\cite{Guillera3} and Straub~\cite{Straub}.

Further, we introduce a new idea for proving some congruences
modulo $\Phi_n(q)$. In many instances in this paper,
the congruences $\sum_{k=0}^{(n-1)/2}a_{n,k}\equiv 0\pmod{\Phi_n(q)}$
are proved by simply showing $a_k+a_{(n-1)/2-k}\equiv 0\pmod{\Phi_n(q)}$
(instead of, say, evaluating certain infinite series at roots of unity
which was illustrated in \cite{GuoZu}).

The proofs of Theorems~\ref{thm:third} and
\ref{thm:fourth} in Section~\ref{sec:proofs34} again are done by showing
a more general identity but otherwise are accomplished in a slightly
different way. All the proofs of Theorems~\ref{thm:first}--\ref{thm:fourth}
in Sections~\ref{sec:proofs} and \ref{sec:proofs34},
and of the further results from Section~\ref{sec:more},
are based on Watson's $_8\phi_7$ transformation formula. We also confirm
a three-parametric $q$-congruence conjecture in Section~\ref{sec:three-para}
based on a quadratic transformation formula of Rahman. Further,
in Section~\ref{sec:cubic} we deduce some $q$-congruences from a
cubic transformation formula of Gasper and Rahman. Similarly, in
Section~\ref{sec:quartic} we deduce some $q$-congruences from a
quartic transformation formula of Gasper and Rahman.
The $q$-supercongruences in Section~\ref{sec:qcong12phi11} are proved
similarly but are derived using a new ${}_{12}\phi_{11}$ transformation formula.
Since the latter formula is of independent interest, its derivation is
given in the Appendix. It is also shown there how a special case
of the ${}_{12}\phi_{11}$ transformation formula can be utilized to
obtain a generalization of Rogers' linearization
formula for the continuous $q$-ultraspherical polynomials.
In Section~\ref{sec:q-dixon} some $q$-supercongruences are deduced
from the $q$-Dixon summation. In Section~\ref{sec:irs} we deduce
$q$-super congruences --most of them only conjectural-- from a double series
transformation of Ismail, Rahman and Suslov.
Finally, in Section~\ref{sec:concl},
some concluding remarks are given and some related conjectures for
further study are proposed. For example, we conjecture that
the congruence \eqref{eq;3rd-noa} still holds modulo $[n]\Phi_n(q)^3$ for
$n\equiv 2\pmod{3}$.

\section{Proofs of Theorems \ref{thm:first} and
\ref{thm:second} }\label{sec:proofs}
We first give the following lemma.
\begin{lemma}\label{lem:2.1}
Let $n$ be a positive odd integer.
Then, for $0\leqslant k\leqslant (n-1)/2$, we have
\begin{equation*}
\frac{(aq;q^2)_{(n-1)/2-k}}{(q^2/a;q^2)_{(n-1)/2-k}}
\equiv (-a)^{(n-1)/2-2k}\frac{(aq;q^2)_k}{(q^2/a;q^2)_k} q^{(n-1)^2/4+k}
\pmod{\Phi_n(q)}.
\end{equation*}
\end{lemma}
\begin{proof}Since $q^n\equiv 1\pmod{\Phi_n(q)}$, we have
\begin{align}\label{aqcong}
\frac{(aq;q^2)_{(n-1)/2} }{(q^2/a;q^2)_{(n-1)/2}}
&=\frac{(1-aq)(1-aq^3)\cdots (1-aq^{n-2})}
{(1-q^2/a)(1-q^4/a)\cdots (1-q^{n-1}/a)} \notag\\[5pt]
&\equiv \frac{(1-aq)(1-aq^3)\cdots (1-aq^{n-2})}
{(1-q^{2-n}/a)(1-q^{4-n}/a)\cdots (1-q^{-1}/a)}\notag\\[5pt]
&=(-a)^{(n-1)/2}q^{(n-1)^2/4} \pmod{\Phi_n(q)}.
\end{align}
Further, modulo $\Phi_n(q)$, we have
\begin{align*}
\frac{(aq;q^2)_{(n-1)/2-k}}{(q^2/a;q^2)_{(n-1)/2-k}}
&=\frac{(aq;q^2)_{(n-1)/2} }{(q^2/a;q^2)_{(n-1)/2}}
\frac{(1-q^{n+1-2k}/a)(1-q^{n+3-2k}/a)\cdots
(1-q^{n-1}/a)}{(1-aq^{n-2k})(1-aq^{n+2-2k})\cdots (1-aq^{n-2})}
\\[5pt]
&\equiv \frac{(aq;q^2)_{(n-1)/2} }{(q^2/a;q^2)_{(n-1)/2}}
\frac{(1-q^{1-2k}/a)(1-q^{3-2k}/a)\cdots
(1-q^{-1}/a)}{(1-aq^{-2k})(1-aq^{2-2k})\cdots (1-aq^{-2})},
\end{align*}
which in combination with \eqref{aqcong} establishes the assertion.
\end{proof}

We now use the above lemma to prove the following result which was originally
conjectured by the first author and Zudilin~\cite[Conj.~5.6]{GuoZu}.
\begin{theorem}
Let $n\equiv3\pmod 4$ be a positive integer. Then
\begin{equation*}
\sum_{k=0}^{(n-1)/2}\frac{(aq, q/a;q^2)_k  (q^2;q^4)_k}
{(aq^2, q^2/a;q^2)_k (q^4;q^4)_k} q^{2k}\equiv 0\pmod{\Phi_n(q)}.
\end{equation*}
\end{theorem}
\begin{proof}By Lemma~\ref{lem:2.1}, we have
\begin{align*}
\frac{(aq, q/a;q^2)_{(n-1)/2-k}  }{(aq^2, q^2/a;q^2)_{(n-1)/2-k}}
&\equiv  \frac{(aq, q/a;q^2)_{k}  }{(aq^2, q^2/a;q^2)_{k}} q^{(n-1)^2/2+2k}
\pmod{\Phi_n(q)}, \\
\intertext{and}
\frac{(q^2;q^4)_{(n-1)/2-k}  }{(q^4;q^4)_{(n-1)/2-k}}
&\equiv (-1)^{(n-1)/2-2k}\frac{(q^2;q^4)_{k}  }{(q^4;q^4)_{k}} q^{(n-1)^2/2+2k}
\pmod{\Phi_n(q^2)}.
\end{align*}
Noticing that $q^n\equiv 1\pmod{\Phi_n(q)}$ and, for odd $n$,
$\Phi_n(q^2)=\Phi_n(q)\Phi_n(-q)$, we get
\begin{equation*}
\frac{(aq, q/a;q^2)_{(n-1)/2-k}  (q^2;q^4)_{(n-1)/2-k} q^{n-1-2k}}
{(aq^2, q^2/a;q^2)_{(n-1)/2-k} (q^4;q^4)_{(n-1)/2-k}}
\equiv -\frac{(aq, q/a;q^2)_k  (q^2;q^4)_k q^{2k}}
{(aq^2, q^2/a;q^2)_k (q^4;q^4)_k}   \pmod{\Phi_n(q)}
\end{equation*}
for any positive integer $n$ with $n\equiv 3\pmod 4$ and
$0\leqslant k\leqslant (n-1)/2$. This completes the proof of the theorem.
\end{proof}

Similarly, we can prove that the third $q$-congruence in \cite[Conj. 5.2]{GuoZu}
is true modulo $\Phi_n(q)$ and is therefore further true modulo $[n]$
(again as in the proof of Theorem~\ref{thm:first}).

We shall establish the following two-parameter generalization
of Theorem~\ref{thm:first}.
\begin{theorem}\label{thm:2.2}
Let $n$ be a positive odd integer. Then, modulo $\Phi_n(q)(1-aq^n)(a-q^n)$,
\begin{align}
&\sum_{k=0}^{(n-1)/2} [4k+1]\frac{(aq, q/a, bq;q^2)_k (q;q^2)_k^3}
{(aq^2, q^2/a, q^2/b;q^2)_k (q^2;q^2)_k^3}\left(\frac{q}{b}\right)^k \notag\\[5pt]
&\quad\equiv [n]q^{(1-n)/2}
\sum_{k=0}^{(n-1)/2}\frac{(aq, q/a, q/b, q;q^2)_k}
{(q^2/b;q^2)_k (q^2;q^2)_k^3}q^{2k}. \label{eq:qab}
\end{align}
\end{theorem}

\begin{proof}
For $a=q^{-n}$ or $a=q^n$ , the left-hand side of \eqref{eq:qab} is equal to
\begin{align}
&\sum_{k=0}^{(n-1)/2} [4k+1]\frac{(q^{1-n}, q^{1+n}, bq;q^2)_k (q;q^2)_k^3}
{(q^{2-n}, q^{2+n}, q^2/b;q^2)_k (q^2;q^2)_k^3}
\left(\frac{q}{b}\right)^k \notag\\[5pt]
&\quad={}_{8}\phi_{7}\!
\left[\begin{array}{cccccccc}
q,& q^{\frac{5}{2}}, & -q^{\frac{5}{2}}, & bq,  & q,  & q^{1-n}, & q^{1+n}, & q  \\
  & q^{\frac{1}{2}}, & -q^{\frac{1}{2}}, & q^2/b, & q^2, & q^{2+n}, & q^{2-n},  & q^2
\end{array};q^2,\, \frac{q}{b}
\right].  \label{eq:a-q-n}
\end{align}
By Watson's $_8\phi_7$ transformation formula \eqref{eq:8phi7},
we can rewrite the right-hand side of \eqref{eq:a-q-n} as
\begin{align}
&\lim_{z\to 1}\frac{(q^3, q, q^{1+n}, zq^{1-n};q^2)_\infty}
{(q^{2+n}, q^{2-n}, q^2, z;q^2)_\infty}
\,{}_{4}\phi_{3}\!\left[\begin{array}{c}
q/b,\,q^{1-n},\, q^{1+n},\, q \\
q^2,\ q^2/b,\ q^2
\end{array};q^2,\, q^2
\right] \notag \\[5pt]
&\quad =\frac{(q;q^2)_{(n+1)/2} (q^{1-n};q^2)_{(n-1)/2}}
{(q^{2-n};q^2)_{(n+1)/2} (q^2;q^2)_{(n-1)/2}}
\sum_{k=0}^{(n-1)/2}\frac{(q^{1-n}, q^{1+n}, q/b, q;q^2)_k}
{(q^2/b;q^2)_k (q^2;q^2)_k^3}q^{2k}.  \label{eq:fraction}
\end{align}
It is easy to see that the fraction before the sum on the right-hand
side of \eqref{eq:fraction} is equal to $[n]q^{(1-n)/2}$. This proves that
the congruence \eqref{eq:qab} holds modulo $1-aq^n$ or $a-q^n$.

Moreover, by Lemma~\ref{lem:2.1}, it is easy to see that, modulo $\Phi_n(q)$,
the $k$-th and $((n-1)/2-k)$-th terms on the left-hand side of \eqref{eq:qab}
cancel each other, i.e.,
\begin{align*}
&\frac{[2n-4k-1](aq, q/a, bq;q^2)_{(n-1)/2-k} (q;q^2)_{(n-1)/2-k}^3}
{(aq^2, q^2/a, q^2/b;q^2)_{(n-1)/2-k} (q^2;q^2)_{(n-1)/2-k}^3}
\left(\frac{q}{b}\right)^{(n-1)/2-k} \notag\\[5pt]
&\quad\equiv -[4k+1]\frac{(aq, q/a, bq;q^2)_k (q;q^2)_k^3}
{(aq^2, q^2/a, q^2/b;q^2)_k (q^2;q^2)_k^3}\left(\frac{q}{b}\right)^k
\pmod{\Phi_n(q)}.
\end{align*}
When the left-hand side of \eqref{eq:qab} has an odd number of factors,
the central term will remain. This happens when $n=4l+1$ for some positive
integer $l$, and in this case the central term has index $k=l$ and one
directly sees that $[4k+1]=[n]$ is a factor of the summand. In total, this
proves that the left-hand side of \eqref{eq:qab} is congruent to $0$ modulo
$\Phi_n(q)$, and therefore the congruence \eqref{eq:qab} also holds modulo
$\Phi_n(q)$. Since $\Phi_n(q)$, $1-aq^n$ and $a-q^n$ are relatively prime
polynomials, the proof of \eqref{eq:qab} is complete.
\end{proof}

\begin{proof}[Proof of Theorem~\ref{thm:first}]
The limits of the denominators on both sides of \eqref{eq:qab} as $a\to 1$
are relatively prime to $\Phi_n(q)$, since $0\leqslant k\leqslant (n-1)/2$.
On the other hand, the limit of $(1-aq^n)(a-q^n)$ as $a\to1$ has the factor
$\Phi_n(q)^2$. Thus, the limiting case $a,b\to1$ of \eqref{eq:qab} gives the
following congruence
\begin{equation}
\sum_{k=0}^{(n-1)/2}[4k+1]\frac{(q;q^2)_k^6}{(q^2;q^2)_k^6} q^k
\equiv [n]q^{(1-n)/2} \sum_{k=0}^{(n-1)/2} \frac{(q;q^2)_k^4}
{(q^2;q^2)_k^4} q^{2k} \pmod{\Phi_n(q)^3},  \label{6-1}
\end{equation}
which also implies that
\begin{equation}
\sum_{k=0}^{n-1}[4k+1]\frac{(q;q^2)_k^6}{(q^2;q^2)_k^6} q^k
\equiv [n]q^{(1-n)/2} \sum_{k=0}^{(n-1)/2}
\frac{(q;q^2)_k^4}{(q^2;q^2)_k^4} q^{2k} \pmod{\Phi_n(q)^3},  \label{6-2}
\end{equation}
since $(q;q^2)_k^6/(q^2;q^2)_k^6\equiv 0\pmod{\Phi_n(q)^3}$ for $k$
in the range $(n-1)/2<k\leqslant n-1$.
It remains to show that the above two congruences are still
true modulo $[n]$, or equivalently,
\begin{subequations}\label{eq:4k+1,both}
\begin{align}
\sum_{k=0}^{(n-1)/2}[4k+1]\frac{(q;q^2)_k^6}{(q^2;q^2)_k^6} q^k
&\equiv 0\pmod{[n]}, \label{eq:4k+1}\\
\intertext{and}
\sum_{k=0}^{n-1}[4k+1]\frac{(q;q^2)_k^6}{(q^2;q^2)_k^6} q^k
&\equiv 0\pmod{[n]}. \label{eq:4k+1,b}
\end{align}
\end{subequations}

For $n>1$, let $\zeta\ne1$ be an $n$-th root of unity, not
necessarily primitive. That is, $\zeta$ is a primitive root of unity
of odd degree $d\mid n$.  Let $c_q(k)$ denote the $k$-th term on the
left-hand side of the congruences in \eqref{eq:4k+1,both}, i.e.,
\begin{equation*}
c_q(k)=[4k+1]\frac{(q;q^2)_k^6}{(q^2;q^2)_k^6} q^k=
[4k+1]
\begin{bmatrix}2k\\k\end{bmatrix}^6
\frac{q^{k}}{(-q;q)_k^{12}}.
\end{equation*}
The congruences \eqref{6-1} and \eqref{6-2} with $n=d$ imply that
\begin{equation*}
\sum_{k=0}^{(d-1)/2}c_\zeta(k)=\sum_{k=0}^{d-1}c_\zeta(k)=0.
\end{equation*}
Observe that
\begin{equation*}
\frac{c_\zeta(\ell d+k)}{c_\zeta(\ell d)}
=\lim_{q\to\zeta}\frac{c_q(\ell d+k)}{c_q(\ell d)}
=c_\zeta(k).
\end{equation*}
We have
\begin{equation*}
\sum_{k=0}^{n-1}c_\zeta(k)=\sum_{\ell=0}^{n/d-1}\sum_{k=0}^{d-1}c_\zeta(\ell d+k)
=\sum_{\ell=0}^{n/d-1}c_\zeta(\ell d) \sum_{k=0}^{d-1}c_\zeta(k)=0,
\end{equation*}
and
\begin{equation*}
\sum_{k=0}^{(n-1)/2}c_\zeta(k)
=\sum_{\ell=0}^{(n/d-3)/2} c_\zeta(\ell d)
\sum_{k=0}^{d-1}c_\zeta(k)+\sum_{k=0}^{(d-1)/2}c_\zeta((n-d)/2+k)=0,
\end{equation*}
which means that the sums $\sum_{k=0}^{n-1}c_q(k)$ and $\sum_{k=0}^{(n-1)/2}c_q(k)$
are both divisible by the cyclotomic polynomial $\Phi_d(q)$.
Since this is true for any divisor $d>1$ of $n$, we conclude that they
are divisible by
\begin{equation*}
\prod_{d\mid n,\, d>1}\Phi_d(q)=[n],
\end{equation*}
thus establishing \eqref{eq:4k+1,both}.
\end{proof}

\begin{proof}[Proof of Theorem~\ref{thm:second}]
Similarly as in the proof of Theorem~\ref{thm:first},
letting  $a\to1$ and $b\to \infty$ in \eqref{eq:qab}, we obtain
\begin{equation*}
\sum_{k=0}^{(n-1)/2}(-1)^k [4k+1]\frac{(q;q^2)_k^5}{(q^2;q^2)_k^5} q^{k^2+k}
\equiv [n]q^{(1-n)/2}
\sum_{k=0}^{(n-1)/2} \frac{(q;q^2)_k^3}{(q^2;q^2)_k^3} q^{2k} \pmod{\Phi_n(q)^3},
\end{equation*}
which also implies that
\begin{equation*}
\sum_{k=0}^{n-1}(-1)^k [4k+1]\frac{(q;q^2)_k^5}{(q^2;q^2)_k^5} q^{k^2+k}
\equiv [n]q^{(1-n)/2}
\sum_{k=0}^{(n-1)/2} \frac{(q;q^2)_k^3}{(q^2;q^2)_k^3} q^{2k} \pmod{\Phi_n(q)^3}.
\end{equation*}
Along the same lines as in the proof of Theorem~\ref{thm:first},
we can show that
\begin{equation*}
\sum_{k=0}^{n-1}(-1)^k [4k+1]\frac{(q;q^2)_k^5}{(q^2;q^2)_k^5} q^{k^2+k}
\equiv
\sum_{k=0}^{(n-1)/2}(-1)^k [4k+1]\frac{(q;q^2)_k^5}{(q^2;q^2)_k^5}
q^{k^2+k}\equiv 0\pmod{[n]}.
\end{equation*}
Combining the above congruences, we are led to \eqref{eq:second-0} and
\eqref{eq:second}.
\end{proof}

\section{Proofs of Theorems \ref{thm:third} and
\ref{thm:fourth}}\label{sec:proofs34}
We shall prove the following common generalization of Theorems~\ref{thm:third}
and \ref{thm:fourth}.
\begin{theorem}\label{thm:third-new}
Let $n$ and $d$ be positive integers with $d\geqslant 3$ and $\gcd(n,d)=1$. Then
\begin{equation}
\sum_{k=0}^{n-1} [2dk+1]\frac{(aq,q/a;q^d)_k (q;q^d)_k^4}
{(aq^d,q^d/a;q^d)_k (q^d;q^d)_k^4} q^{(2d-3)k}
\equiv
\begin{cases} 0  \pmod{\Phi_n(q)^2}, &\text{if $n\equiv -1\pmod d$,}\\[10pt]
 0 \pmod{\Phi_n(q)}, &\text{otherwise.}
\end{cases} \label{eq:qd2}
\end{equation}
\end{theorem}
\begin{proof}
Let $\alpha$ and $j$ be integers. Since
\begin{equation*}
(1-q^{\alpha n-dj+d-1})(1-q^{\alpha  n+dj-d+1})+(1-q^{dj-d+1})^2
q^{\alpha  n-dj+d-1}=(1-q^{\alpha  n})^2
\end{equation*}
and $1-q^{\alpha n}\equiv 0\pmod{\Phi_n(q)}$, we obtain
\begin{equation*}
(1-q^{\alpha  n-dj+d-1})(1-q^{\alpha  n+dj-d+1})\equiv -(1-q^{dj-d+1})^2
q^{\alpha  n-dj+d-1}\pmod{\Phi_n(q)^2}.
\end{equation*}
It follows that
\begin{equation*}
(q^{1-\alpha  n},q^{1+\alpha  n};q^d)_k \equiv (q;q^d)_k^2 \pmod{\Phi_n(q)^2}.
\end{equation*}
Similarly, we have
\begin{equation*}
(q^{d-\alpha  n},q^{d+\alpha  n};q^d)_k \equiv (q^d;q^d)_k^2 \pmod{\Phi_n(q)^2}.
\end{equation*}

Since $\gcd(n,d)=1$, we know that there exists a positive integer $\alpha<d$ such that $\alpha n\equiv 1\pmod{d}$. Then by \cite[Appendix~(III.18)]{GR} (i.e.,
\eqref{eq:8phi7} with $f=q^{-\alpha n}$), modulo $\Phi_n(q)^2$, the left-hand side
of \eqref{eq:qd2} is congruent to
\begin{align*}
&\sum_{k=0}^{(\alpha n-1)/d}[2dk+1]\frac{(aq, q/a, q,q, q^{1+\alpha
n}, q^{1-\alpha n};q^d)_k }
{(aq^d, q^d/a, q^d, q^d, q^{d-\alpha n}, q^{d+\alpha n};q^d)_k} q^{(2d-3)k}  \\[5pt]
&\quad=\frac{(q^{d+1}, q^{d-\alpha n-1};q^d)_{(\alpha n-1)/d}}{(q^d, q^{d-\alpha n};q^d)_{(\alpha n-1)/d}}\,
 {}_{4}\phi_{3}\!\left[\begin{array}{c}
q^{d-1},\,q,\,q^{1+\alpha n},\, q^{1-\alpha n} \\
q^2,\, aq^d,\, q^d/a,
\end{array};q^d,\,q^d
\right].
\end{align*}
It is clear that $(q^{d+1};q^d)_{(\alpha n-1)/d}$ in the numerator has the
factor $1-q^{\alpha n}$ and is therefore divisible by $\Phi_n(q)$, while the
denominator is coprime with $\Phi_n(q)$. This proves \eqref{eq:qd2}
for the second case.

Furthermore, if $n\equiv -1\pmod{d}$, then, modulo $\Phi_n(q)^2$, the left-hand
side of \eqref{eq:qd2} is congruent to
\begin{align*}
&\sum_{k=0}^{((d-1)n-1)/d} [2dk+1]
\frac{(aq, q/a, q, q, q^{1+(d-1)n}, q^{1-(d-1)n};q^d)_k}
{(aq^d, q^d/a, q^d, q^d, q^{d-(d-1)n}, q^{d+(d-1)n};q^d)_k} q^{(2d-3)k} \\[5pt]
&\quad =\frac{(q^{d+1}, q^{d-(d-1)n-1};q^d)_{((d-1)n-1)/d}}
{(q^d, q^{d-(d-1)n};q^d)_{((d-1)n-1)/d}}\,
 {}_{4}\phi_{3}\!\left[\begin{array}{c}
q^{d-1},\,q,\,q^{1+(d-1)n},\, q^{1-(d-1)n} \\
q^2,\, aq^d,\, q^d/a,
\end{array};q^d,\,q^d
\right].
\end{align*}
It is easy to see that this time the numerator has the factor
$(1-q^{(d-1)n})(1-q^{(2-d)n})$ and is therefore divisible by $\Phi_n(q)^2$,
and again the denominator is coprime with $\Phi_n(q)$.
This proves \eqref{eq:qd2} for the first case.
\end{proof}

Letting $a=1$ and $d=3$ in \eqref{eq:qd2}, we get
\begin{equation*}
\sum_{k=0}^{n-1} [6k+1]\frac{(q;q^3)_k^6}{(q^3;q^3)_k^6} q^{3k}
\equiv
\begin{cases} 0 \pmod{\Phi_n(q)^2}, &\text{if $n\equiv 2\pmod 3$,} \\[10pt]
 0 \pmod{\Phi_n(q)}, &\text{if $n\equiv 1\pmod 3$.}
\end{cases}
\end{equation*}
Similarly as in the proof of Theorem~\ref{thm:first}, we can prove that
\begin{equation*}
\sum_{k=0}^{n-1} [6k+1]\frac{(q;q^3)_k^6}{(q^3;q^3)_k^6} q^{3k}
\equiv 0\pmod{[n]}.
\end{equation*}
This completes the proof of \eqref{eq;3rd-noa}.

Likewise, taking $a\to 0$ in \eqref{eq:qd2}, we obtain
\begin{equation*}
\sum_{k=0}^{n-1}[2dk+1]\frac{(q;q^d)_k^4}{(q^d;q^d)_k^4}q^{(d-2)k}
\equiv
\begin{cases} 0  \pmod{\Phi_n(q)^2}, &\text{if $n\equiv -1\pmod d$,}\\[10pt]
0 \pmod{\Phi_n(q)}, &\text{otherwise.}
\end{cases}
\end{equation*}
This completes the proof of \eqref{eq:2dk+1}.

It appears that the following generalization with one more parameter $b$
is still true.
\begin{conjecture}\label{conj:2dk+1-new}
Let $n$ and $d$ be positive integers with $d\geqslant 3$ and $\gcd(n,d)=1$. Then
\begin{align*}
&\sum_{k=0}^{n-1} [2dk+1]\frac{(aq,q/a,bq, q/b;q^d)_k (q;q^d)_k^2}
{(aq^d, q^d/a, bq^d, q^d/b;q^d)_k  (q^d;q^d)_k^2} q^{(2d-3)k} \notag \\[5pt]
&\quad\equiv
\begin{cases} 0  \pmod{[n]\Phi_n(q)}, &\text{if $n\equiv -1\pmod d$,}\\[10pt]
 0 \pmod{[n]}, &\text{otherwise.}
\end{cases}
\end{align*}
\end{conjecture}

\section{More $q$-congruences from Watson's transformation}
\label{sec:more}
Throughout this section, $m$ always stands for $n-1$ or $(n+1)/2$.
Note that the special case of \cite[Thm.~4.9]{GuoZu} with $r=-1$,
$d=2$ and $a=1$ gives
\begin{equation*}
\sum_{k=0}^{m}(-1)^k [4k-1]\frac{(q^{-1};q^2)_k^3}{(q^2;q^2)_k^3} q^{k^2+2k}
\equiv [n](-q)^{(n-3)(n+1)/4} \pmod{[n]\Phi_n(q)^2}\quad \text{for odd}\ n>1.
\end{equation*}
In this section, we shall give some similar congruences.
\begin{theorem}\label{thm:4k-1-5th}
Let $n>1$ be a positive odd integer. Then
\begin{align}
&\sum_{k=0}^{m}(-1)^k [4k-1]\frac{(q^{-1};q^2)_k^5}{(q^2;q^2)_k^5} q^{k^2+5k}
\notag \\[5pt]
&\quad\equiv [n](-q)^{(n+1)(n-3)/4}\sum_{k=0}^{(n+1)/2}
\frac{(q^{-1};q^2)_k^2 (q^3;q^2)_k}{(q^2;q^2)_k^3 } q^{3k}
\pmod{[n]\Phi_n(q)^2}.   \label{eq:4k-1-5th}
\end{align}
\end{theorem}
\begin{proof} We first establish the following result:
\begin{align}
&\sum_{k=0}^{(n+1)/2}(-1)^k [4k-1]\frac{(aq^{-1},q^{-1}/a;q^2)_k (q^{-1};q^2)_k^3}
{(aq^2,q^2/a;q^2)_k (q^2;q^2)_k^3} q^{k^2+5k}
\equiv [n](-q)^{(n+1)(n-3)/4} \notag \\[5pt]
&\qquad\quad\times\sum_{k=0}^{(n+1)/2}
\frac{(aq^{-1},q^{-1}/a;q^2)_k (q^3;q^2)_k}{(q^2;q^2)_k^3 }
q^{3k}\pmod{\Phi_n(q)(1-aq^n)(a-q^n)}.
\label{eq:4k-1-5th-a}
\end{align}
For $a=q^{-n}$ or $a=q^n$, the left-hand side of \eqref{eq:4k-1-5th-a}
is equal to
\begin{align}
&\sum_{k=0}^{(n+1)/2}(-1)^k  [4k-1]\frac{(q^{-1-n},q^{-1+n};q^2)_k (q^{-1};q^2)_k^3}
{(q^{2-n}, q^{2+n};q^2)_k (q^2;q^2)_k^3}q^{k^2+5k} \notag\\[5pt]
&\quad=-q^{-1}\sum_{k=0}^{(n+1)/2}(-1)^k \frac{(1-q^{4k-1})
(q^{-1},q^{-1},q^{-1},q^{-1-n},q^{-1+n};q^2)_k }
{(1-q^{-1})(q^2,q^2,q^2, q^{2+n}, q^{2-n};q^2)_k } q^{k^2+5k} .
\label{eq:a-q-nm}
\end{align}
By the limiting case of Watson's transformation formula \eqref{eq:8phi7-2},
we can rewrite the right-hand side of \eqref{eq:a-q-nm} as
\begin{equation*}
\frac{-q^{-1}(q, q^3;q^2)_\infty}{(q^{2+n}, q^{2-n};q^2)_\infty}
\sum_{k=0}^{(n+1)/2}\frac{(q^{-1-n},q^{-1+n};q^2)_k (q^3;q^2)_k}
{(q^2;q^2)_k^3 } q^{3k}.
\end{equation*}
It is easy to see that
\begin{equation*}
\frac{-q^{-1}(q, q^3;q^2)_\infty}{(q^{2+n}, q^{2-n};q^2)_\infty}
=\frac{-q^{-1}(q;q^2)_{(n+1)/2}}{(q^{2-n};q^2)_{(n+1)/2}}
=[n](-q)^{(n+1)(n-3)/4}.
\end{equation*}
This proves that the congruence \eqref{eq:4k-1-5th-a}
holds modulo $(1-aq^n)(a-q^n)$.

On the other hand, by Lemma~\ref{lem:2.1}, for $0\leqslant k\leqslant (n+1)/2$,
we have
\begin{align}
\frac{(aq^{-1};q^2)_{(n+1)/2-k}}{(q^2/a;q^2)_{(n+1)/2-k}}
&=\frac{(1-aq^{-1})(aq;q^2)_{(n-1)/2-k}}
{(1-q^{n+1-2k}/a)(q^2/a;q^2)_{(n-1)/2-k}} \notag \\[5pt]
&\equiv (-a)^{(n-1)/2-2k}\frac{(1-aq^{-1})(aq;q^2)_k}
{(1-q^{1-2k}/a)(q^2/a;q^2)_k} q^{(n-1)^2/4+k} \notag\\[5pt]
&=(-a)^{(n+1)/2-2k}\frac{(aq^{-1};q^2)_k}{(q^2/a;q^2)_k} q^{(n-1)^2/4+3k-1}
\pmod{\Phi_n(q)}.  \label{eq:mod-phi}
\end{align}
It follows that the $k$-th and $((n+1)/2-k)$-th terms on the left-hand side of
\eqref{eq:4k-1-5th-a} cancel each other modulo $\Phi_n(q)$. When the respective
sum has an odd number of factors, the central term will remain. This happens
when $n=4l-1$ for some positive integer $l$, and in this case the central term
has index $k=l$ and one directly sees that $[4k-1]=[n]$ is a factor of the
summand. In total, this proves that
the congruence \eqref{eq:4k-1-5th-a} also holds modulo $\Phi_n(q)$.
Since the polynomials $\Phi_n(q)$, $1-aq^n$ and $a-q^n$ are coprime with one
another, the proof of \eqref{eq:4k-1-5th-a} is complete.

Letting $a\to 1$ in \eqref{eq:4k-1-5th-a}, one sees that the congruence
\eqref{eq:4k-1-5th} holds modulo $\Phi_n(q)^3$ by noticing that
$(q^{-1};q^2)_5^k/(q^2;q^2)_k^5\equiv 0\pmod{\Phi_n(q)^3}$ for
$(n+1)/2<k\leqslant n-1$. Along the same lines of the proof of
Theorem~\ref{thm:first}, we can prove that
\begin{equation*}
\sum_{k=0}^{m}(-1)^k [4k-1]\frac{(q^{-1};q^2)_k^5}{(q^2;q^2)_k^5}
q^{k^2+5k}\equiv 0\pmod{[n]},
\end{equation*}
i.e., that the congruence \eqref{eq:4k-1-5th} holds modulo $[n]$. Since
$\lcm(\Phi_n(q)^3,[n])=[n]\Phi_n(q)^2$, the proof of the theorem is complete.
\end{proof}

\begin{corollary}We have
\begin{equation*}
\sum_{k=0}^{(p+1)/2}(-1)^k (4k-1)\frac{(-\frac{1}{2})_k^5}{k!^5}
\equiv (-1)^{(p+1)/2} p \sum_{k=0}^{(p+1)/2}
\frac{(-\frac{1}{2})_k^2 (\frac{3}{2})_k}{k!^3} \pmod{p^3}.
\end{equation*}
\end{corollary}

\begin{theorem}\label{thm:4k-1-4th}
Let $n>1$ be a positive odd integer. Then
\begin{equation}
\sum_{k=0}^{m}[4k-1]\frac{(q^{-1};q^2)_k^4}{(q^2;q^2)_k^4} q^{4k}
\equiv 0 \pmod{[n]\Phi_n(q)^2}. \label{eq:th4}
\end{equation}
\end{theorem}
\begin{proof}
We first establish the following congruence:
\begin{equation}
\sum_{k=0}^{m}[4k-1]\frac{(aq^{-1}, q^{-1}/a;q^2)_k(q^{-1};q^2)_k^2}
{(aq^2, q^2/a;q^2)_k(q^2;q^2)_k^2} q^{4k} \equiv 0
\pmod{\Phi_n(q)(1-aq^n)(a-q^n)}.  \label{eq:4th-a}
\end{equation}
Letting $q\mapsto q^2$ and $c\to 0$ followed by $a=b=q^{-1}$, $d=q^{-1-n}$ and
$e=q^{-1+n}$ in the limiting case of Watson's transformation formula
\eqref{eq:8phi7-2}, we obtain
\begin{align}
&\sum_{k=0}^{m}[4k-1]\frac{(q^{-1-n}, q^{-1+n};q^2)_k (q^{-1};q^2)_k^2}
{(q^{2-n}, q^{2+n};q^2)_k(q^2;q^2)_k^2} q^{4k}  \notag\\[5pt]
&\quad=\frac{-q^{-1}(q, q^3;q^2)_\infty}{(q^{2+n}, q^{2-n};q^2)_\infty}
\sum_{k=0}^{(n+1)/2}\frac{(q^{-1-n},q^{-1+n};q^2)_k }{(q^2;q^2)_k^2} q^{4k}.
\label{eq:4k-1-4th}
\end{align}
By the $q$-Chu-Vandermonde summation formula \cite[Appendix~(II.6)]{GR},
for odd $n>1$, we have
\begin{equation}
\sum_{k=0}^{m}\frac{(q^{-1-n},q^{-1+n};q^2)_k }{(q^2;q^2)_k^2} q^{2k}=
\frac{(q^{3-n};q^2)_{(n+1)/2}}{(q^2;q^2)_{(n+1)/2}}q^{(n^2-1)/2}=0.
\label{eq:q-Chu}
\end{equation}
Letting $q\mapsto q^{-1}$ in the above equality, we see that the summation on
the right-hand side of \eqref{eq:4k-1-4th} is equal to $0$.
This proves that the congruence \eqref{eq:4th-a} holds modulo $(1-aq^n)(a-q^n)$.

On the other hand, similarly as before, by \eqref{eq:mod-phi} one sees that
the sum of the $k$-th and $((n+1)/2-k)$-th terms on the left-hand side of
\eqref{eq:4th-a} are congruent to $0$ modulo $\Phi_n(q)$
(and also, when the respective sum has an odd number of factors, i.e.,
when $n=4l-1$ for some positive integer $l$,
then the remaining central term has index $k=l$ and one directly sees
that $[4k-1]=[n]$ is a factor of the summand).
This thus proves that
the congruence \eqref{eq:4th-a} is also true modulo $\Phi_n(q)$.
This completes the proof of \eqref{eq:4th-a}.

Let $c_q(k)$ denote the $k$-th term on the left-hand side of \eqref{eq:4th-a}.
In the same vein as in the proof of Theorem~\ref{thm:first}, we can further
prove that
\begin{equation}
\sum_{k=0}^{n-1}c_q(k)=\sum_{k=0}^{(n+1)/2}c_q(k)\equiv 0\pmod{[n]}. \label{eq:yzyz}
\end{equation}
Thus, we have proved that
\begin{equation}
\sum_{k=0}^{m}[4k-1]\frac{(aq^{-1}, q^{-1}/a;q^2)_k(q^{-1};q^2)_k^2}
{(aq^2, q^2/a;q^2)_k(q^2;q^2)_k^2} q^{4k} \equiv 0 \pmod{[n](1-aq^n)(a-q^n)}.
\label{eq:xyxy}
\end{equation}
The parts of the denominators in \eqref{eq:xyxy} which contain the parameter
$a$ are the factors of $(aq^2, q^2/a;q^2)_{(n+1)/2}$ or
$(aq^2, q^2/a;q^2)_{n-1}$.
Their limits as $a\to1$ are relatively prime to $\Phi_n(q)$.
On the other hand, the limit of $(1-aq^n)(a-q^n)$ as $a\to1$ has the factor
$\Phi_n(q)^2$.
Therefore, the limiting case $a\to 1$ of the congruence \eqref{eq:xyxy}
reduces to \eqref{eq:th4} modulo $\Phi_n(q)^3$. But the congruences
\eqref{eq:yzyz} are still true when $a=1$ which implies that
the congruence \eqref{eq:th4} holds modulo $[n]$.
This completes the proof of the theorem.
\end{proof}

It appears that the congruence conditions stated in Theorem~\ref{thm:4k-1-4th}
and its extension in \eqref{eq:4th-a} can be strengthened:
\begin{conjecture}\label{conj:4k-1}
Let $n>1$ be a positive odd integer. Then
\begin{align*}
\sum_{k=0}^{m}[4k-1]\frac{(aq^{-1};q^2)_k(q^{-1}/a;q^2)_k(q^{-1};q^2)_k^2}
{(aq^2;q^2)_k(q^2/a;q^2)_k(q^2;q^2)_k^2} q^{4k}
&\equiv 0 \pmod{[n]^2(1-aq^n)(a-q^n)},  \\
\intertext{and}
\sum_{k=0}^{m}[4k-1]\frac{(q^{-1};q^2)_k^4}{(q^2;q^2)_k^4} q^{4k}
&\equiv 0 \pmod{[n]^4}.
\end{align*}
\end{conjecture}

\begin{theorem}\label{thm:4k+1-6th}
Let $n>3$ be a positive odd integer. Then
\begin{equation}
\sum_{k=0}^{(n+1)/2}[4k+1]\frac{(q^{-1};q^2)_k^2 (q;q^2)_k^2}
{(q^4;q^2)_k^2 (q^2;q^2)_k^2} q^{4k}
\equiv 0 \pmod{[n]\Phi_n(q)^2}. \label{eq:4k+1-6th}
\end{equation}
\end{theorem}
\begin{proof}The proof is similar to that of Theorem \ref{thm:4k-1-4th}.
We first establish
\begin{equation}
\sum_{k=0}^{(n+1)/2}[4k+1]\frac{(aq^{-1}, q^{-1}/a;q^2)_k(q;q^2)_k^2}
{(aq^4, q^4/a;q^2)_k(q^2;q^2)_k^2} q^{4k} \equiv 0
\pmod{\Phi_n(q)(1-aq^n)(a-q^n)}  \label{eq:6th-a}
\end{equation}
for odd $n>1$.
Letting $q\mapsto q^2$ and $c\to 0$ followed by $a=b=q$, $d=q^{-1-n}$ and
$e=q^{-1+n}$ in \eqref{eq:8phi7-2}, we obtain
\begin{align}
&\sum_{k=0}^{(n+1)/2}[4k+1]\frac{(q^{-1-n}, q^{-1+n};q^2)_k (q;q^2)_k^2}
{(q^{4-n}, q^{4+n};q^2)_k(q^2;q^2)_k^2} q^{4k}  \notag\\[5pt]
&\quad=\frac{(q^3, q^5;q^2)_\infty}{(q^{4+n}, q^{4-n};q^2)_\infty}
\sum_{k=0}^{(n+1)/2}\frac{(q^{-1-n},q^{-1+n};q^2)_k }{(q^2;q^2)_k^2} q^{4k}.
\label{eq:4k-1-6th}
\end{align}
As we have already mentioned in the proof of Theorem \ref{thm:4k-1-4th},
the summation on
the right-hand side of \eqref{eq:4k-1-6th} is equal to $0$ by the
$q\mapsto q^{-1}$ case of \eqref{eq:q-Chu}.
Thus, we have proved that the congruence \eqref{eq:6th-a} holds modulo
$(1-aq^n)(a-q^n)$.

On the other hand, similarly as before, by \eqref{eq:mod-phi} one sees that
the sum of the $k$-th and $((n-1)/2-k)$-th terms on the left-hand side of
\eqref{eq:6th-a} are congruent to $0$ modulo $\Phi_n(q)$
for $0\leqslant k\leqslant (n-1)/2$. Moreover, the summand for $k=(n+1)/2$ on
the right-hand side of \eqref{eq:4k-1-6th} is clearly congruent
to $0$ modulo $\Phi_n(q)$
because of the factor $(q;q^2)_{(n+1)/2}$ in the numerator.
This proves that the congruence \eqref{eq:6th-a} is also true
modulo $\Phi_n(q)$.
The proof of \eqref{eq:6th-a} is completed.

For $n>3$, we have $(n+3)/2<n$ and so the denominator of the left-hand side
of \eqref{eq:6th-a} is
relatively prime to $\Phi_n(q)$ when taking the limit as $a\to 1$.
Therefore, the congruence \eqref{eq:4k+1-6th}
holds modulo $\Phi_n(q)^3$ for $n>3$ by taking $a\to 1$ in \eqref{eq:6th-a}.
On the other hand, it is also easy to see that
the congruence \eqref{eq:4k+1-6th} holds modulo $\Phi_3(q)$ for $n=3$.
Let $c_q(k)$ denote the $k$-th term on the left-hand side of \eqref{eq:6th-a}.
Similarly to the proof of Theorem~\ref{thm:first}, we can further prove that
\begin{equation*}
\sum_{k=0}^{n-2}c_q(k)=\sum_{k=0}^{(n+1)/2}c_q(k)\equiv 0\pmod{[n]}.
\end{equation*}
This proves \eqref{eq:4k+1-6th}.
\end{proof}

We conjecture that the following generalization of \eqref{eq:6th-a}
and Theorem \ref{thm:4k+1-6th} is still true.
\begin{conjecture}Let $n>3$ be a positive odd integer. Then
\begin{align*}
\sum_{k=0}^{(n+1)/2}[4k+1]\frac{(aq^{-1};q^2)_k(q^{-1}/a;q^2)_k(q;q^2)_k^2}
{(aq^4;q^2)_k(q^4/a;q^2)_k(q^2;q^2)_k^2} q^{4k}
&\equiv 0 \pmod{[n]\Phi_n(q)(1-aq^n)(a-q^n)};  \\
\intertext{in particular,}
\sum_{k=0}^{(n+1)/2}[4k+1]\frac{(q^{-1};q^2)_k^2 (q;q^2)_k^2}
{(q^4;q^2)_k^2 (q^2;q^2)_k^2} q^{4k}
&\equiv 0 \pmod{[n]\Phi_n(q)^3}.
\end{align*}
\end{conjecture}

Analogously, letting $q\mapsto q^2$ and $c\to 0$ followed by
$a=q$, $b=q^{-1}$, $d=q^{-1-n}$ and
$e=q^{-1+n}$ in \eqref{eq:8phi7-2}, we can prove the following result:
\begin{equation}
\sum_{k=0}^{(n+1)/2}[4k+1]\frac{(aq^{-1};q^2)_k (q^{-1}/a;q^2)_k(q^{-1},q;q^2)_k}
{(aq^4;q^2)_k(q^4/a;q^2)_k(q^4,q^2;q^2)_k} q^{6k}
\equiv 0 \pmod{\Phi_n(q)(1-aq^n)(a-q^n)}.  \label{eq:7th-a}
\end{equation}
We label the limiting case $a\to 1$ as the following theorem.
\begin{theorem}Let $n>3$ be a positive odd integer. Then
\begin{equation}
\sum_{k=0}^{(n+1)/2}[4k+1]\frac{(q^{-1};q^2)_k^3 (q;q^2)_k}
{(q^4;q^2)_k^3 (q^2;q^2)_k} q^{6k}
\equiv 0 \pmod{\Phi_n(q)^3}.  \label{eq:8th}
\end{equation}
Moreover, if $\gcd(n,3)=1$, then the above congruence holds modulo
$[n]\Phi_n(q)^2$.
\end{theorem}

It seems that the following generalization of
\eqref{eq:7th-a} and \eqref{eq:8th} still holds.
\begin{conjecture}Let $n>3$ be a positive odd integer. Then
\begin{align*}
\sum_{k=0}^{(n+1)/2}[4k+1]\frac{(aq^{-1};q^2)_k (q^{-1}/a;q^2)_k(q^{-1},q;q^2)_k}
{(aq^4;q^2)_k(q^4/a;q^2)_k(q^4,q^2;q^2)_k} q^{6k}
&\equiv 0 \pmod{\Phi_n^2 (q)(1-aq^n)(a-q^n)};  \\
\intertext{in particular,}
\sum_{k=0}^{(n+1)/2}[4k+1]\frac{(q^{-1};q^2)_k^3 (q;q^2)_k}
{(q^4;q^2)_k^3 (q^2;q^2)_k} q^{6k}
&\equiv 0 \pmod{\Phi_n(q)^4}.
\end{align*}
\end{conjecture}

We also have the following similar result.
\begin{theorem}\label{thm:4k+1-7th}
Let $n>1$ be a positive odd integer. Then
\begin{equation}
\sum_{k=0}^{(n-1)/2}[4k+1]\frac{(q^{-1};q^2)_k (q;q^2)_k^3}
{(q^4;q^2)_k (q^2;q^2)_k^3} q^{2k}
\equiv 0 \pmod{[n]^3}. \label{eq:4k+1-7th}
\end{equation}
\end{theorem}
\begin{proof}It is easy to see by induction on $N$ that
\begin{equation*}
\sum_{k=0}^{N}[4k+1]\frac{(q^{-1};q^2)_k (q;q^2)_k^3}
{(q^4;q^2)_k (q^2;q^2)_k^3} q^{2k}
=\frac{(q;q^2)_{N} (q;q^2)_{N+1}^3}{(1-q)^3(q^4;q^2)_{N} (q^2;q^2)_N^3}.
\end{equation*}
Putting $N=(n-1)/2$ in the above identity and using \eqref{eq:q-bino-1/2},
we get
\begin{align}
\sum_{k=0}^{(n-1)/2}[4k+1]\frac{(q^{-1};q^2)_k (q;q^2)_k^3}
{(q^4;q^2)_k (q^2;q^2)_k^3} q^{2k}
&=\frac{[n]^3  (q;q^2)_{(n-1)/2}^4}
{(q^4;q^2)_{(n-1)/2} (q^2;q^2)_{(n-1)/2}^3} \notag \\[5pt]
&=\frac{[n]^3 (1+q)}{[n+1](-q;q)_{(n-1)/2}^8}
\begin{bmatrix}n-1\\(n-1)/2\end{bmatrix}^4. \label{eq:4k+1-8th}
\end{align}
The proof then follows from the fact that $\gcd([n],[n+1])=1$
and $\gcd([n],(-q;q)_N)=1$ (for odd $n$).
\end{proof}
Let $n=p$ and $q=1$ in \eqref{eq:4k+1-8th}. Using Fermat's little theorem,
we immediately obtain the following conclusion.
\begin{corollary}We have
\begin{equation*}
\sum_{k=0}^{(p-1)/2}(4k+1)\frac{(-\frac{1}{2})_k (\frac{1}{2})_k^3 }{(k+1)!k!^3}
\equiv 2p^3 \pmod{p^4}.
\end{equation*}
\end{corollary}

We end this section with the following conjecture, which is similar
to Conjecture \ref{conj:2dk+1-new}. Similarly as in the proof of
Theorem \ref{thm:third-new}, we can confirm it for $b=1$.
\begin{conjecture}
Let $n$ and $d$ be positive integers with $d\geqslant 3$ and $\gcd(n,d)=1$. Then
\begin{align*}
&\sum_{k=0}^{n-1} [2dk-1]\frac{(aq^{-1},q^{-1}/a,bq^{-1},
q^{-1}/b;q^d)_k (q^{-1};q^d)_k^2}
{(aq^d, q^d/a, bq^d, q^d/b;q^d)_k  (q^d;q^d)_k^2} q^{(2d+3)k} \notag \\[5pt]
&\quad\equiv
\begin{cases} 0  \pmod{[n]\Phi_n(q)}, &\text{if $n\equiv 1\pmod d$,}\\[10pt]
 0 \pmod{[n]}, &\text{otherwise.}
\end{cases}
\end{align*}
\end{conjecture}

\section{Proof of a three-parametric $q$-congruence from
a quadratic transformation of Rahman}
\label{sec:three-para}
In this section, we confirm a three-parametric $q$-congruence conjecture
of the first author and Zudilin \cite[Conj.~4.6]{GuoZu}.
\begin{theorem}Let $n$ be a positive odd integer. Then,
modulo $[n](1-aq^n)(a-q^n)$,
\begin{equation}
\sum_{k=0}^{M}[3k+1]\frac{(aq,q/a,q;q^2)_k (q/b,q/c,bc;q)_k}
{(aq,q/a,q;q)_k (bq^2,cq^2,q^3/bc;q^2)_k}  q^k
\equiv\frac{(bcq,q^2/b,q^2/c;q^2)_{(n-1)/2}}
{(q^3/bc,bq^2,cq^2;q^2)_{(n-1)/2}}[n],  \label{eq:3-par}
\end{equation}
where $M=n-1$ or $(n-1)/2$.
\end{theorem}

The congruence \eqref{eq:3-par} modulo $(1-aq^n)(a-q^n)$ has already
been proved by the first author and Zudilin \cite[Thm.~4.7]{GuoZu}.
Moreover,  the congruence \eqref{eq:3-par} with $c=1$  was established
in \cite[Thm.~4.8]{GuoZu}. Therefore, it remains to be proven
that \eqref{eq:3-par} holds modulo $[n]$, i.e.,
\begin{equation}
\sum_{k=0}^{M}[3k+1]\frac{(aq,q/a,q;q^2)_k (q/b,q/c,bc;q)_k}
{(aq,q/a,q;q)_k (bq^2,cq^2,q^3/bc;q^2)_k}  q^k
\equiv 0 \pmod{[n]}.  \label{eq:3-par-2}
\end{equation}
\begin{proof}
We need to use a quadratic transformation formula of Rahman \cite{Rahman}
(see also \cite[Eq.~(3.8.13)]{GR}):
\begin{align}
&\sum_{k=0}^\infty \frac{(1-aq^{3k})(a,d,aq/d;q^2)_k (b,c, aq/bc;q)_k}
{(1-a)(aq/d,d,q;q)_k (aq^2/b,aq^2/c,bcq;q^2)_k}q^k   \notag\\[5pt]
&\quad=\frac{(aq^2,bq,cq,aq^2/bc;q^2)_\infty}{(q,aq^2/b,aq^2/c,bcq;q^2)_\infty}
\,{}_{3}\phi_{2}\!\left[\begin{array}{c}
b,\,c,\,aq/bc\\
dq,\,aq^2/d
\end{array};q^2,\, q^2
\right], \label{eq:1-aq3k}
\end{align}
provided $d$ or $aq/d$ is not of the form $q^{-2n}$, $n$ a non-negative
integer. It is clear that \eqref{eq:3-par-2} is true for $n=1$. We now
suppose that $n>1$.
Let $a=q^{1-n}$ in \eqref{eq:1-aq3k} and then we further set $d=aq$ and
replace $b$ and $c$ with $q/b$ and $q/c$, respectively.
Then the left-hand side of \eqref{eq:1-aq3k}
terminates at $k=(n-1)/2$, and the right-side of \eqref{eq:1-aq3k}
vanishes because the numerator contains the factor
$(q^{3-n};q^2)_\infty$. Namely, we have
\begin{equation*}
\sum_{k=0}^{(n-1)/2}\frac{1-q^{3k+1-n}}{1-q^{1-n}}
\frac{(aq,q^{1-n}/a,q^{1-n};q^2)_k (q/b,q/c,bcq^{-n};q)_k}
{(aq,q^{1-n}/a,q;q)_k (bq^{2-n},cq^{2-n},q^3/bc;q^2)_k}  q^k
=0.
\end{equation*}
Since $q^{n}\equiv 1\pmod{\Phi_n(q)}$, we immediately get
\begin{equation*}
\sum_{k=0}^{(n-1)/2}[3k+1]\frac{(aq,q/a,q;q^2)_k (q/b,q/c,bc;q)_k}
{(aq,q/a,q;q)_k (bq^2,cq^2,q^3/bc;q^2)_k}  q^k
\equiv 0 \pmod{\Phi_n(q)}.
\end{equation*}
Finally, the proof of \eqref{eq:3-par-2} is completely analogous to that
of Theorem~\ref{thm:first} (more precisely, to the proofs of \eqref{eq:4k+1}
and \eqref{eq:4k+1,b}).
\end{proof}

Letting $a,b\to 1$ and $c\to 0$ in \eqref{eq:3-par}, we obtain
\begin{equation}
\sum_{k=0}^{M}[3k+1]\frac{(q;q^2)_k^3 q^{-{\binom{k+1}2} } }
{(q;q)_k^2 (q^2;q^2)_k}
\equiv q^{(1-n)/2}[n] \pmod{[n]\Phi_n(q)^2}   \label{eq:q-div-WZ-1}
\end{equation}
(see also \cite{Guo4}), while letting $a,b\to 1$ and $c\to -1$ in
\eqref{eq:3-par}, we get
\begin{equation}
\sum_{k=0}^{M}[3k+1]\frac{(q;q^2)_k^3 (-1;q)_k q^k}{(q;q)_k^3 (-q^2,-q^3;q^2)_k }
\equiv \frac{1+q}{1+q^n}[n] \pmod{[n]\Phi_n(q)^2}.   \label{eq:q-div-WZ-2}
\end{equation}
It is easy to see that both \eqref{eq:q-div-WZ-1} and \eqref{eq:q-div-WZ-2}
are $q$-analogues of the following supercongruence:
\begin{equation}
\sum_{k=0}^{M} (3k+1)\frac{(\frac{1}{2})_k^3}{k!^3}2^{2k}
\equiv p\pmod{p^3}, \label{eq:div-1}
\end{equation}
where $M=p-1$ or $M=(p-1)/2$.
The congruence \eqref{eq:div-1} with $M=(p-1)/2$ was first proved
by Guillera and Zudilin \cite{GuZu} using the WZ (Wilf--Zeilberger) method.
For $M=p-1$, Hu \cite{Hu} proved an even stronger congruence, namely
\begin{equation}
\sum_{k=0}^{p-1} (3k+1)\frac{(\frac{1}{2})_k^3}{k!^3}2^{2k}
\equiv p\pmod{p^4}.
\end{equation}
Also the $M=(p-1)/2$ case of \eqref{eq:div-1} can be extended
to a congruence modulo $p^4$, namely
\begin{equation}
\label{eq:E1}
\sum_{k=0}^{(p-1)/2} (3k+1)\frac{(\frac{1}{2})_k^3}{k!^3}2^{2k}
\equiv p+(-1)^{(p-1)/2}2p^3E_{p-3}\pmod{p^4},
\end{equation}
where $E_{p-3}$ is the $(p-3)$-th Euler number, which was conjectured
by Sun~\cite[Conj.~5.1(ii)]{Sun4} and recently proved by Mao and
Zhang~\cite{MZ}.

Here we would like to propose a supercongruence similar to \eqref{eq:div-1}.
\begin{conjecture}\label{conj:3k-1}
We have
\begin{equation}
\sum_{k=0}^{(p+1)/2} (3k-1)\frac{(-\frac{1}{2})_k^2
(\frac{1}{2})_k}{k!^3}2^{2k}
\equiv p\pmod{p^3}. \label{eq:3k-1}
\end{equation}
\end{conjecture}
Unfortunately, we were not able to find any $q$-analogue of \eqref{eq:3k-1},
even for the
simple case modulo $p$.

Moreover, letting $a\to 1$, $b\to -1$ and $c\to 0$ in \eqref{eq:3-par}, we get
\begin{equation}
\sum_{k=0}^{M}(-1)^k [3k+1]\frac{(q;q^2)_k^3 (-q;q)_k q^{-{\binom{k+1}2}}}
{(q;q)_k^3 (-q^2;q^2)_k}
\equiv [n](-q)^{(1-n)/2}\pmod{[n]\Phi_n(q)^2}, \label{eq:q-Zudilin-44}
\end{equation}
while letting $a\to 1$ and $b,c\to 0$ in \eqref{eq:3-par}, we arrive at
\begin{equation}
\sum_{k=0}^{M}(-1)^k [3k+1]\frac{(q;q^2)_k^3}{(q;q)_k^3}
\equiv [n](-q)^{(n-1)^2/4} \pmod{[n]\Phi_n(q)^2}. \label{eq:q-Zudilin-33}
\end{equation}
It is worth mentioning that both \eqref{eq:q-Zudilin-44} and
\eqref{eq:q-Zudilin-33} are $q$-analogues of the following supercongruence
due to Guillera and Zudilin \cite{GuZu}:
\begin{equation}
\sum_{k=0}^{(p-1)/2} (-1)^k(3k+1)\frac{(\frac{1}{2})_k^3}{k!^3} 2^{3k}
\equiv p(-1)^{(p-1)/2}\pmod{p^3}.  \label{eq:q-Zudilin-55}
\end{equation}
The congruence \eqref{eq:q-Zudilin-33} with $M=n-1$ was first established
by the first author \cite{Guo4} using the $q$-WZ method. The congruence
\eqref{eq:q-Zudilin-44} is new.

Sun~\cite[Conj.~5.1(ii)]{Sun4} conjectured that
\begin{equation}
\label{eq:E2}
\sum_{k=0}^{p-1}(3k+1)\frac{(\frac{1}{2})_k^3}{k!^3} 2^{3k}
\equiv p(-1)^{(p-1)/2}+p^3E_{p-3}\pmod{p^4}
\end{equation}
(where $E_{p-3}$ is again the $(p-3)$-th Euler number), for any prime $p>3$,
which was confirmed by Chen, Xie and He~\cite{CXH}.

Motivated by Conjecture \ref{conj:3k-1}, \eqref{eq:E1} and \eqref{eq:E2},
we would like to raise the following problems.
\begin{problem}Is there a ``$3k-1$ version'' of the supercongruence
\eqref{eq:q-Zudilin-55}?
\end{problem}
\begin{problem}
Are there any $q$-analogues of hypergeometric supercongruences involving
Euler numbers as in \eqref{eq:E1} or \eqref{eq:E2}?
\end{problem}
\section{Some $q$-congruences from a cubic transformation
of Gasper and Rahman}\label{sec:cubic}
Gasper and Rahman \cite{GR0} (see also \cite[Eq.\ (3.8.18)]{GR}) obtained
the following cubic transformation:
\begin{align}
&\sum_{k=0}^\infty \frac{1-acq^{4k}}{1-ac}
\frac{(a,q/a;q)_k (ac;q)_{2k} (d,acq/d;q^3)_k }
{(cq^3,a^2cq^2;q^3)_k (q;q)_{2k} (acq/d,d;q)_k}q^k \notag\\[5pt]
&\quad=\frac{(acq^2,acq^3,d/ac,dq/ac,adq,aq,q^2/a,dq^2/a;q^3)_\infty}
{(q,q^2,dq,dq^2,a^2cq^2,cq^3,dq/a^2c,d/c;q^3)_\infty}
\notag\allowdisplaybreaks[0]\\[5pt]
&\quad\quad+\frac{d(a,q/a,acq;q)_\infty (q^3,d,acq/d,d^2q^2/ac;q^3)_\infty}
{ac (q,d,acq/d;q)_\infty (cq^3,a^2cq^2,d/c,dq/a^2c;q^3)_\infty}
\,{}_{2}\phi_{1}\!\left[\begin{array}{c}
\!\! d/c,\,dq/a^2c\\
d^2q^2/ac
\end{array};q^3,\, q^3
\right]. \label{eq:1-acq4k}
\end{align}
This transformation for $d=0$ becomes a summation formula and has been used
by the first author and Zudilin \cite{GuoZu} to prove the following
$q$-congruence: modulo $[n](1-aq^n)(a-q^n)$,
\begin{equation}
\sum_{k=0}^{M}[8k+1] \frac{(aq, q/a;q^2)_k (q;q^2)_{2k}}
{(q^2;q^2)_{2k}(aq^6, q^6/a;q^6)_k } q^{2k^2}
\equiv q^{-(n-1)/2}[n]\biggl(\frac{-3}{n}\biggr),
\end{equation}
where $M=n-1$ or $(n-1)/2$, $\big(\frac{\cdot}{\cdot}\big)$
is the Jacobi--Kronecker symbol, and $\gcd(n,6)=1$.

In this section, we shall deduce some $q$-congruences from
\eqref{eq:1-acq4k} with $d\neq 0$.
\begin{theorem}\label{thm:8k+1-1}
Let $n>1$ be a positive integer coprime with $6$. Then
\begin{align}
&\sum_{k=0}^{(n-1)/2}[8k+1] \frac{(aq, q/a;q^2)_k (q;q^2)_{2k} (q,q^2;q^6)_k}
{(q^2;q^2)_{2k}(aq^6, q^6/a;q^6)_k (q;q)_{2k} } q^{2k} \notag\\[10pt]
&\quad\equiv
\begin{cases}  0 \pmod{\Phi_n(q)}, &\text{if $n\equiv 1\pmod 6,$}\\[10pt]
0  \pmod{\Phi_n(q)(1-aq^n)(a-q^n)}, &\text{if $n\equiv 5\pmod 6$.}
\end{cases} \label{8k+1-1}
\end{align}
\end{theorem}
\begin{proof}We specialize \eqref{eq:1-acq4k} by taking
$q\mapsto q^2$, then $c=q^{1-n}/a$, $d=q$, and replace $a$ by $aq$.
The right-hand side of the resulting identity vanishes,
because the numerator of the first fraction contains  the factors
$(q^{5-n};q^6)_\infty$ and $(q^{7-n};q^6)_\infty$,
and the numerator of the second fraction has the factor $(q^{3-n};q^2)_\infty$.
At the same time, the left-hand side terminates at $k=(n-1)/2$
(in fact, much earlier, at $k=\lfloor n/4\rfloor$), for the summand
involves the term $(q^{1-n};q^2)_{2k}$. This proves that
\begin{equation*}
\sum_{k=0}^{(n-1)/2}\frac{1-q^{1-n+8k}}{1-q^{1-n}}
\frac{(aq, q/a;q^2)_k (q^{1-n};q^2)_{2k} (q,q^{2-n};q^6)_k}
{(q^2;q^2)_{2k}(aq^{6-n}, q^{6-n}/a;q^6)_k (q^{2-n};q^2)_k (q;q^2)_k } q^{2k}
=0.
\end{equation*}
Since $q^n\equiv  1\pmod{\Phi_n(q)}$, we deduce that the congruence
\eqref{8k+1-1} holds modulo $\Phi_n(q)$.

We now assume that $n\equiv 5\pmod{6}$. It remains to show that
\begin{equation*}
\sum_{k=0}^{(n-1)/2}[8k+1] \frac{(aq, q/a;q^2)_k (q;q^2)_{2k} (q,q^2;q^6)_k}
{(q^2;q^2)_{2k}(aq^6, q^6/a;q^6)_k (q;q)_{2k} } q^{2k}
\equiv 0\pmod{(1-aq^n) (a-q^n)},
\end{equation*}
or equivalently,
\begin{equation}
\sum_{k=0}^{(n-1)/2}[8k+1] \frac{(q^{1-n}, q^{1+n};q^2)_k (q;q^2)_{2k}
(q,q^2;q^6)_k}{(q^2;q^2)_{2k}(q^{6-n}, q^{6+n};q^6)_k (q;q)_{2k} } q^{2k}
=0. \label{eq:special}
\end{equation}
This identity again follows straightforwardly from \eqref{eq:1-acq4k}
by setting $q\mapsto q^2$, $a=q^{1-n}$, $c=q^n$, and $d=q$. In fact,
the left-hand side of the resulting identity terminates at $k=(n-1)/2$
and is therefore just the left-hand side of \eqref{eq:special}.
On the other hand, the right-hand side of the resulting identity
vanishes because the first fraction is equal to
\begin{equation*}
\frac{(q^5,q^7,1,q^2,q^{4-n},q^{3-n},q^{3+n},q^{4+n};q^6)_\infty}
{(q^2,q^4,q^3,q^5,q^{6-n},q^{6+n},q^{1+n},q^{1-n};q^6)_\infty}
=0,
\end{equation*}
and the second fraction is equal to
\begin{equation*}
\frac{(q^{1-n},q^{1+n},q^3;q^2)_\infty (q^6,q,q^2,q^5;q^6)_\infty}
{(q^2,q,q^2;q^2)_\infty (q^{6+n},q^{6-n},q^{1-n},q^{1+n};q^6)_\infty}
=0.
\end{equation*}
This completes the proof of the theorem.
\end{proof}

Similarly as in the proof of Theorem \ref{thm:8k+1-1}, we can prove
the following result.
\begin{theorem}\label{thm:8k+1-2}
Let $n>1$ be a positive integer coprime with $6$. Then
\begin{align}
&\sum_{k=0}^{(n-1)/2}[8k+1] \frac{(aq, q/a;q^2)_k (q;q^2)_{2k} (q^{-1},q^4;q^6)_k}
{(q^2;q^2)_{2k}(aq^6, q^6/a;q^6)_k (q^{-1},q^4;q^2)_{k} } q^{2k} \notag\\[10pt]
&\quad\equiv
\begin{cases}  0 \pmod{\Phi_n(q)(1-aq^n)(a-q^n)},
&\text{if $n\equiv 1\pmod 6,$}\\[10pt]
0  \pmod{\Phi_n(q)}, &\text{if $n\equiv 5\pmod 6$.}
\end{cases}
\end{align}
\end{theorem}

Letting $a\to 1$ in Theorems \ref{thm:8k+1-1} and \ref{thm:8k+1-2}, we obtain
\begin{corollary}Let $n>1$ be a positive integer coprime with $6$. Then
\begin{equation*}
\sum_{k=0}^{(n-1)/2}[8k+1] \frac{(q;q^2)_k^2 (q;q^2)_{2k} (q,q^2;q^6)_k}
{(q^2;q^2)_{2k}(q^6;q^6)_k^2 (q;q)_{2k} } q^{2k}
\equiv
\begin{cases}  0 \pmod{\Phi_n(q)}, &\text{if $n\equiv 1\pmod 6,$}\\[10pt]
0  \pmod{\Phi_n(q)^3 }, &\text{if $n\equiv 5\pmod 6$,}
\end{cases}
\end{equation*}
and
\begin{equation*}
\sum_{k=0}^{(n-1)/2}\![8k+1] \frac{(q;q^2)_k^2 (q;q^2)_{2k} (q^{-1},q^4;q^6)_k}
{(q^2;q^2)_{2k}(q^6;q^6)_k^2 (q^{-1},q^4;q^2)_{k} } q^{2k}
\equiv
\begin{cases}  0 \!\pmod{\Phi_n(q)^3},
&\!\text{if $n\equiv 1\!\!\pmod 6,$}\\[10pt]
0  \!\pmod{\Phi_n(q) }, &\!\text{if $n\equiv 5\!\!\pmod 6$.}
\end{cases}
\end{equation*}
\end{corollary}

We shall also prove the following results.
\begin{theorem}
Let $n$ be a positive integer coprime with $6$. Then
\begin{equation}
\sum_{k=0}^{M}[8k+1] \frac{(q;q^2)_k^2 (q;q^2)_{2k} (q^{-3};q^6)_k}
{(q^2;q^2)_{2k}(q^6;q^6)_k (q^{-3},q^6;q^2)_{k} } q^{2k}
\equiv 0 \pmod{[n]^2}, \label{eq:8k+1:conj}
\end{equation}
where $M=n-1$ or $(n-1)/2$.
\end{theorem}
\begin{proof}
It is easy to see by induction on $N$ that
\begin{align}
&\sum_{k=0}^{N} \frac{[8k+1] (q;q^2)_k^2 (q;q^2)_{2k} (q^{-3};q^6)_k}
{(q^2;q^2)_{2k}(q^6;q^6)_k (q^{-3},q^6;q^2)_{k} } q^{2k} \notag \\[5pt]
&\quad=\frac{[4N+1][4N+3](q;q^2)_{N+1} (q;q^2)_{2N} (q^3;q^6)_N }
{(1-q^3)(q^6;q^6)_N (q^2;q^2)_{2N} (q^6;q^2)_N} \notag\\[5pt]
&\quad=\frac{[4N+1][4N+3][2N+1][2][4]}
{[3][2N+2][2N+4](-q;q)_N^2 (-q;q)_{2N}^2 (-q^3;q^3)_N^2}
\begin{bmatrix}2N\\N\end{bmatrix}\begin{bmatrix}4N\\2N\end{bmatrix}
\begin{bmatrix}2N\\N\end{bmatrix}_{q^3}.  \label{eq:8k+1-qbino}
\end{align}
Note that $\frac{1}{[N+1]}
\begin{bmatrix}\begin{smallmatrix}2N\\N\end{smallmatrix}\end{bmatrix}$
is the well-known $q$-Catalan number (see \cite{FH}), a polynomial in $q$.
Hence, the $q$-binomial coefficient
$\begin{bmatrix}\begin{smallmatrix}2N\\N\end{smallmatrix}\end{bmatrix}$
is divisible by $[N+1]$, so is
$\begin{bmatrix}\begin{smallmatrix}2N\\N\end{smallmatrix}\end{bmatrix}_{q^3}$
if $N+1$ is coprime with $3$.
It is also not difficult to prove that
$\begin{bmatrix}\begin{smallmatrix}4N\\2N\end{smallmatrix}\end{bmatrix}$
is divisible by $[N+1]$ whenever $N+1$ is coprime with $6$.
Therefore, putting $N=n-1$ in \eqref{eq:8k+1-qbino}, we can prove that
the right-hand side is congruent to $0$ modulo $[n]^2$.
Similarly, taking $N=(n-1)/2$ in \eqref{eq:8k+1-qbino}, we arrive at
the same conclusion. This time one $[n]$ comes from $[2N+1]$ and another $[n]$
comes from
$\begin{bmatrix}\begin{smallmatrix}4N\\2N\end{smallmatrix}\end{bmatrix}$.
\end{proof}

\section{Some $q$-congruences from a quartic transformation
of Gasper and Rahman}\label{sec:quartic}

Gasper and Rahman \cite{GR0} (see also \cite[Ex.\ 3.33]{GR}) also obtained
the following quartic transformation:
\begin{align}
&\sum_{k=0}^\infty \frac{1-a^2b^2 q^{5k-2}}{1-a^2b^2/q^2}
\frac{(a,b;q)_k (ab/q,ab,abq;q^3)_k (a^2b^2/q^2;q^4)_k}
{(ab^2q^2,a^2bq^2;q^4)_k (abq,ab,ab/q;q^2)_k (q;q)_k} q^k  \notag\\[5pt]
&\quad=\frac{(aq,b;q)_\infty (-abq;q^2)_\infty}
{(q;q)_\infty  (b,ab^2q^2,a^2bq^2;q^4)_\infty}
\,{}_{1}\phi_{1}\!\left[\begin{array}{c}
\!\! a\\
aq^4
\end{array};q^4,\, bq^4
\right]. \label{eq:quartic}
\end{align}
In this section, we shall deduce two congruences from the quartic
transformation \eqref{eq:quartic}.

\begin{theorem}Let $n$ be a positive integer with $n\equiv 5,7\pmod{8}$. Then
\begin{equation}
\sum_{k=0}^{(n-1)/2}[10k+2]
\frac{(q;q)_{2k}(q;q^2)_{3k}(q^2;q^8)_k}
{(q^9, q^8;q^8)_k(q;q^2)_{2k}(q^5;q^4)_k(q^2;q^2)_k}q^{2k}
\equiv 0\pmod{\Phi_n(q)}.  \label{cong:quartic-1}
\end{equation}
\end{theorem}
\begin{proof}Replacing $q$ by $q^2$, $a$ by $q^{1-n}$ and
$b$ by $q^{2-n}$ in \eqref{eq:quartic}, we see that the left-hand
side terminates at $k=(n-1)/2$, while the right-hand side vanishes.
(Note that we cannot make such a replacement
if $n\equiv 1,3\pmod{8}$.) Namely, we have
\begin{equation*}
\sum_{k=0}^{(n-1)/2}\frac{1-q^{10k+2-4n}}{1-q^{2-4n}}
\frac{(q^{1-n};q)_{2k}(q^{1-2n};q^2)_{3k}(q^{2-4n};q^8)_k}
{(q^{9-3n}, q^{8-3n};q^8)_k(q^{1-2n};q^2)_{2k}(q^{5-2n};q^4)_k(q^2;q^2)_k}q^{2k}
=0.
\end{equation*}
Since $q^n\equiv 1\pmod{\Phi_n(q)}$, we immediately obtain
\eqref{cong:quartic-1} from the above identity.
\end{proof}

It is not difficult to see that the congruence \eqref{cong:quartic-1}
can also be derived from the following quartic summation formula of
Gasper \cite{Gasper} (see also \cite[Ex. 3.30]{GR}):
\begin{align*}
&\sum_{k=0}^\infty \frac{1-aq^{5k}}{1-a}
\frac{(a,b;q)_k (q/b,q^2/b,q^3/b;q^3)_k (a^2b^2/q^2;q^4)_k}
{(q^4,aq^4/b;q^4)_k (abq,ab,ab/q;q^2)_k (q^3/ab^2;q)_k} q^k  \notag\\[5pt]
&\quad+\frac{ab^3(aq,bq,1/b;q)_\infty (a^2b^2q^2;q^4)_\infty}
{q^2(ab,q^3/ab^2;q)_\infty (ab/q;q^2)_\infty (q^4,ab^3/q^2,aq^4/b;q^4)_\infty}
\,{}_{1}\phi_{1}\!\left[\begin{array}{c}
\!\! a^2b^2/q^2\\
a^2b^2q^2
\end{array};q^4,\, ab^3q^2
\right]\\[5pt]
&\quad\quad=\frac{(aq,ab^2/q^2;q)_\infty }
{(ab;q)_\infty (ab/q;q^2)_\infty (aq^4/b,ab^3/q^2;q^4)_\infty}.
\end{align*}

\begin{theorem}Let $n$ be a positive integer with $n\equiv 5,7\pmod{8}$. Then
\begin{equation}
\sum_{k=0}^{(n-1)/2}[10k+4]
\frac{(q,q^3;q^2)_{k}(q^2;q^2)_{3k}(q^4;q^8)_k}
{(q^{11}, q^9;q^8)_k(q^2;q^2)_{2k}(q^6;q^4)_k(q^2;q^2)_k}q^{2k}
\equiv 0\pmod{\Phi_n(q)}.  \label{cong:quartic-2}
\end{equation}
\end{theorem}
\begin{proof}Replacing $q$ by $q^2$, $a$ by $q^{1-n}$ and
$b$ by $q^3$ in \eqref{eq:quartic}, we see that the left-hand
side again terminates at $k=(n-1)/2$, while the right-hand side vanishes.
That is,
\begin{equation*}
\sum_{k=0}^{(n-1)/2}\frac{1-q^{10k+4-2n}}{1-q^{4-2n}}
\frac{(q^{1-n},q^3;q^2)_{k}(q^{2-n};q^2)_{3k}(q^{4-2n};q^8)_k}
{(q^{11-n}, q^{9-2n};q^8)_k(q^{2-n};q^2)_{2k}(q^{6-n};q^4)_k(q^2;q^2)_k}q^{2k}
=0.
\end{equation*}
The proof of \eqref{cong:quartic-2} then follows from the above identity
and the fact $q^n\equiv 1\pmod{\Phi_n(q)}$.
\end{proof}

We have the following two related conjectures.

\begin{conjecture}The congruence \eqref{cong:quartic-1} is still true
modulo $\Phi_n(q)^2$ for $n\equiv 5\pmod{8}$.
In particular,
if $p\equiv 5\pmod{8}$, then
\begin{equation*}
\sum_{k=0}^{(p-1)/2}\frac{(5k+1) (1)_{2k}(\frac{1}{2})_{3k}
(\frac{1}{4})_{k} }{32^k (\frac{9}{8})_k (1)_k^2
(\frac{1}{2})_{2k} (\frac{5}{4})_k }
\equiv 0\pmod{p^2}.
\end{equation*}
\end{conjecture}

\begin{conjecture}The congruence \eqref{cong:quartic-2} is still true
modulo $\Phi_n(q)^3$. In particular,
if $p\equiv 5,7\pmod{8}$, then
\begin{equation*}
\sum_{k=0}^{(p-1)/2}\frac{(5k+2) (\frac{1}{2})_k^2 (1)_{3k} }
{8^k (\frac{11}{8})_k (\frac{9}{8})_k (1)_{2k} (1)_k }
\equiv 0\pmod{p^3}.
\end{equation*}

\end{conjecture}

\section{Some $q$-congruences from a new
${}_{12}\phi_{11}$ transformation}\label{sec:qcong12phi11}
In this section, we shall deduce some $q$-congruences from Theorem~\ref{newtf},
a new $_{12}\phi_{11}$ transformation formula, whose proof we give in the
appendix. Although all of the $q$-congruences are modulo $\Phi_n(q)$, the
$q=1$ cases sometimes can be generalized to supercongruences modulo higher
powers (see Conjectures~\ref{conj:7.5} and \ref{conj:7.6} in the next section).
\begin{theorem}Let $n\equiv 1\pmod{3}$ be a positive integer and $n>1$. Then
\begin{subequations}
\begin{align}
\sum_{k=0}^{(n-1)/3} [6k+1]\frac{(q;q^3)_k^6
(q^3;q^3)_{2k}}{(q^3;q^3)_k^6 (q^2;q^3)_{2k}} q^{2k}
&\equiv 0 \pmod{\Phi_n(q)},  \label{eq:6k+1} \\[5pt]
\sum_{k=0}^{(n-1)/3} [6k+1]\frac{(q;q^3)_k^4
(q^3;q^3)_{2k}}{(q^3;q^3)_k^4 (q^2;q^3)_{2k}}
&\equiv 0 \pmod{\Phi_n(q)},  \label{eq:6k+1-2} \\
\intertext{and}
\sum_{k=0}^{(n-1)/3} [6k+1]\frac{(q;q^3)_k^2
(q^3;q^3)_{2k}}{(q^3;q^3)_k^2 (q^2;q^3)_{2k}} q^{-2k}
&\equiv 0 \pmod{\Phi_n(q)}.  \label{eq:6k+1-3}
\end{align}
\end{subequations}
\end{theorem}
\begin{proof} Replacing $q\mapsto q^3$ and then letting
$a=q^{1-n}$, $b=c=d=q$ in \eqref{12phi11}, we obtain
\begin{equation}
\sum_{k=0}^{(n-1)/3}\frac{1-q^{6k+1-n}}{1-q^{1-n}}\frac{(q^{1-n};q^3)_k^3(q;q^3)_k^3
(q^{3-n};q^3)_{2k}}{(q^{3-n};q^3)_k^3 (q^3;q^3)_k^3 (q^{2-n};q^3)_{2k}} q^{2k}
=0, \label{eq:6k+1-n}
\end{equation}
because the right-hand side of \eqref{12phi11} contains the factor
$(q^{1-n};q^3)_\infty$, which vanishes for $n\equiv 1\pmod{3}$. Since
$q^{n}\equiv 1\pmod{\Phi_n(q)}$, we immediately deduce \eqref{eq:6k+1}
from \eqref{eq:6k+1-n} .

Similarly, if we change $c=q$ to $c\to 0$ in the above procedure, then we
can prove \eqref{eq:6k+1-2}, while if we change $c=d=q$ to $c,d\to 0$ then we
are led to \eqref{eq:6k+1-3}.
\end{proof}

\begin{theorem}Let $n\equiv 2\pmod{3}$ be a positive integer and $n>2$. Then
\begin{subequations}
\begin{align}
\sum_{k=0}^{(n+1)/3} [6k-1]\frac{(q^{-1};q^3)_k^6
(q^3;q^3)_{2k}}{(q^3;q^3)_k^6 (q^{-2};q^3)_{2k}} q^{4k}
\equiv 0 \pmod{\Phi_n(q)},  \label{eq:6k-1} \\[5pt]
\sum_{k=0}^{(n+1)/3} [6k-1]\frac{(q^{-1};q^3)_k^4
(q^3;q^3)_{2k}}{(q^3;q^3)_k^4 (q^{-2};q^3)_{2k}}
\equiv 0 \pmod{\Phi_n(q)},  \label{eq:6k-1-2} \\
\intertext{and}
\sum_{k=0}^{(n+1)/3} [6k-1]\frac{(q^{-1};q^3)_k^2
(q^3;q^3)_{2k}}{(q^3;q^3)_k^2 (q^{-2};q^3)_{2k}} q^{-4k}
\equiv 0 \pmod{\Phi_n(q)}.  \label{eq:6k-1-3}
\end{align}
\end{subequations}
\end{theorem}
\begin{proof} Replacing $q\mapsto q^3$ and then letting  $a=q^{-1-n}$,
$b=c=d=q^{-1}$ in \eqref{12phi11}, we obtain
\begin{equation}
\sum_{k=0}^{(n+1)/3}\frac{1-q^{6k-1-n}}{1-q^{-1-n}}\frac{(q^{-1-n};q^3)_k^3(q;q^3)_k^3
(q^{3-n};q^3)_{2k}}{(q^{3-n};q^3)_k^3 (q^{3};q^3)_k^3 (q^{-2-n};q^3)_{2k}} q^{4k}
=0,\label{eq:6k-1-n}
\end{equation}
because the right-hand side of \eqref{12phi11} contains the factor
$(q^{2-n};q^3)_\infty$, which vanishes for $n\equiv 2\pmod{3}$.
It is easy to see that the denominator of \eqref{eq:6k-1-n} is relatively
prime to $\Phi_n(q)$ for $n>2$.
Therefore, applying $q^{n}\equiv 1\pmod{\Phi_n(q)}$, we obtain the desired
congruence in \eqref{eq:6k-1}.
Similarly (see the proof of \eqref{eq:6k+1-2} and \eqref{eq:6k+1-3}),
we can prove \eqref{eq:6k-1-2} and \eqref{eq:6k-1-3}.
\end{proof}

\section{Some other $q$-congruences from the
$q$-Dixon sum}\label{sec:q-dixon}
By using the $q$-Dixon sum \cite[Eq.~(II.13)]{GR}:
\begin{equation}
{}_{4}\phi_{3}\!\left[\begin{array}{c}
 a,\, -q\sqrt{a},\,b,\,c\\
\!\!-\sqrt{a},\,aq/b,\,aq/c
\end{array};q,\, \frac{q\sqrt{a}}{bc}
\right]
=\frac{(aq,q\sqrt{a}/b,q\sqrt{a}/c,aq/bc;q)_\infty}
{(aq/b,aq/c,q\sqrt{a},q\sqrt{a}/bc;q)_\infty}, \label{eq:q-Dixon}
\end{equation}
 the first author and Zudilin \cite[Thm.~4.12]{GuoZu} proved the
following result.
\begin{equation}
\sum_{k=0}^{(n-1)/2} \frac{(1+aq^{4k+1})(a^2q^2, bq^2, cq^2;q^4)_k }
{(1+aq)(a^2q^4/b, a^2q^4/c, q^4;q^4)_k } \biggl(\frac{aq}{bc}\biggr)^k
\equiv 0\pmod{(1-a^2q^{2n})} \quad\text{for}\ n\equiv 3\pmod{4};
\label{eq:qDixon-1}
\end{equation}
in particular,
\begin{equation}
\sum_{k=0}^{(n-1)/2} \frac{(1+q^{4k+1})\, (q^2;q^4)_k^3}{(1+q)\,(q^4;q^4)_k^3} q^k
\equiv 0\pmod{\Phi_n(q)\Phi_n(-q)}\quad\text{for}\ n\equiv 3\pmod{4}.
\label{eq:qDixon-2}
\end{equation}
They \cite[Conj.~4.13]{GuoZu} also conjectured that the congruence
\eqref{eq:qDixon-2} still holds modulo $\Phi_n(q)^2 \Phi_n(-q)$.

In this section, we shall give further similar congruences from the
$q$-Dixon sum.
\begin{theorem}\label{thm:q-Dixon}
Let $n\equiv 1\pmod{4}$ be a positive integer. Then
\begin{equation}
\sum_{k=0}^{(n+1)/2} \frac{(1+aq^{4k-1}) (a^2/q^2,  b/q^2,  c/q^2;q^4)_k }
{(1+aq^{-1}) (a^2q^4/b, a^2q^4/c, q^4;q^4)_k } \biggl(\frac{aq^7}{bc}\biggr)^k
\equiv 0\pmod{(1-a^2q^{2n})}; \label{eq:qDixon-3}
\end{equation}
in particular,
\begin{equation}
\sum_{k=0}^{(n+1)/2} \frac{(1+q^{4k-1}) (q^{-2};q^4)_k^3}{(1+q)(q^4;q^4)_k^3} q^{7k}
\equiv 0\pmod{\Phi_n(q)\Phi_n(-q)}\quad\text{for odd}\ n>1.
\label{eq:qDixon-4}
\end{equation}
\end{theorem}
\begin{proof}
Letting $q\mapsto q^4$, $a\mapsto a^2/q^2$, $b\mapsto b/q^2$ and
$c\mapsto c/q^2$ in \eqref{eq:q-Dixon}, we get
\begin{equation}
\sum_{k=0}^\infty \frac{(1+aq^{4k-1}) (a^2/q^2,  b/q^2, c/q^2;q^4)_k}
{(1+aq^{-1}) (a^2q^4/b, a^2q^4/c, q^4;q^4)_k} \biggl(\frac{aq^7}{bc}\biggr)^k
=\frac{(a^2q^2, aq^5/b, aq^5/c, a^2q^6/bc;q^4)_\infty}
{(a^2q^4/b, a^2q^4/c, aq^3, aq^7/bc;q^4)_\infty}.
\label{eq:qDixon-33}
\end{equation}
Since $n\equiv 1\pmod 4$, putting $a=\pm q^{-n}$ in \eqref{eq:qDixon-33}
we see that the left-hand side terminates at $k=(n+1)/2$,
while the right-hand side vanishes. This proves \eqref{eq:qDixon-3}.
For $n>1$, taking the limit as $a,b,c\to 1$ in \eqref{eq:qDixon-3}
we are led to \eqref{eq:qDixon-4}.
\end{proof}

We conjecture that the following stronger version of \eqref{eq:qDixon-4}
is also true.
\begin{conjecture} \label{conj:GuoZu}
Let $n\equiv 1\pmod{4}$ be an integer and $n>1$. Then
\begin{equation}
\sum_{k=0}^{(n+1)/2} \frac{(1+q^{4k-1}) (q^{-2};q^4)_k^3}{(1+q)(q^4;q^4)_k^3} q^{7k}
\equiv 0\pmod{\Phi_n(q)^2\Phi_n(-q)}.
\end{equation}
\end{conjecture}

Similarly to the proof of Theorem \ref{thm:q-Dixon}, taking $q\mapsto q^4$,
$a\mapsto a^2q^2$ and $b=c= q^{-2}$ in \eqref{eq:q-Dixon},
we can prove the following result.
\begin{theorem}
Let $n>1$ be a positive odd integer. Then
\begin{equation*}
\sum_{k=0}^{(n-1)/2} \frac{(1+aq^{4k+1}) (a^2q^2,  q^{-2},  q^{-2};q^4)_k }
{(1+aq) (a^2q^8, a^2q^8, q^4;q^4)_k } a^k q^{9k}
\equiv 0\pmod{(1-a^2q^{2n})};
\end{equation*}
in particular,
\begin{equation}
\sum_{k=0}^{(n-1)/2} \frac{(1+q^{4k+1}) (q^2;q^4)_k (q^{-2};q^4)_k^2}
{(1+q)(q^8;q^4)_k^2 (q^4;q^4)_k} q^{9k}
\equiv 0\pmod{\Phi_n(q)\Phi_n(-q)}. \label{eq:eq:qDxion-5}
\end{equation}
\end{theorem}

Note that, for $n\equiv 3\pmod{4}$, we can prove the following
three-parametric congruence:
\begin{equation}
\sum_{k=0}^{(n-1)/2} \frac{(1+aq^{4k+1}) (a^2q^2,  b/q^{2},  c/q^{2};q^4)_k }
{(1+aq) (a^2q^8/b, a^2q^8/c, q^4;q^4)_k } \biggl(\frac{aq^9}{bc}\biggr)^k
\equiv 0\pmod{(1-a^2q^{2n})};
\end{equation}
Besides, for the $q=-1$ case of \eqref{eq:eq:qDxion-5},
it seems that the corresponding congruence can be strengthened as follows.
\begin{conjecture}\label{conj:jk-1}
Let $p\equiv 3\pmod{4}$. Then
\begin{equation*}
\sum_{k=0}^{(p-1)/2}(-1)^k (4k+1) \frac{(\frac{1}{2})_k
(-\frac{1}{2})_k^2}{(k+1)!^2 k!}
\equiv 0\pmod{p^2}.
\end{equation*}
\end{conjecture}

Likewise, performing another set of parameter replacements $q\mapsto q^4$,
$a\mapsto a^2q^2$, $b\mapsto bq^2$, and $c\mapsto c/q^2$ in \eqref{eq:q-Dixon},
we can deduce the following result.
\begin{theorem}
Let $n\equiv 3\pmod{4}$ be a positive odd integer. Then
\begin{equation*}
\sum_{k=0}^{(n-1)/2} \frac{(1+aq^{4k+1}) (a^2q^2,  bq^{2},  c/q^{2};q^4)_k }
{(1+aq) (a^2q^4/b, a^2q^8/c, q^4;q^4)_k } \biggl(\frac{aq^5}{bc}\biggr)^k
\equiv 0\pmod{(1-a^2q^{2n})};
\end{equation*}
in particular,
\begin{equation}
\sum_{k=0}^{(n-1)/2} \frac{(1+q^{4k+1}) (q^2;q^4)_k^2 (q^{-2};q^4)_k}
{(1+q)(q^8;q^4)_k (q^4;q^4)_k^2} q^{5k}
\equiv 0\pmod{\Phi_n(q)\Phi_n(-q)}. \label{eq:eq:qDxion-6}
\end{equation}
\end{theorem}
We have the following conjectures.
\begin{conjecture}For any integer $n\equiv 3\pmod{4}$ and $n>3$, the
congruence \eqref{eq:eq:qDxion-6} still holds modulo $\Phi_n(q)^2 \Phi_n(-q)$.
\end{conjecture}

\begin{conjecture}
Let $n\equiv 3\pmod{4}$ be a positive integer. Then
\begin{equation}
\sum_{k=0}^{(n-1)/2} [4k+1]\frac{ (q^{2};q^4)_k (q^4;q^8)_k}
{(q^4;q^4)_k (q^8;q^8)_k} q^{k}
\equiv 0\pmod{\Phi_n(q)^2\Phi_n(-q)}. \label{4k+1-Dixon-f}
\end{equation}
\end{conjecture}
It is easy to see that the congruence \eqref{4k+1-Dixon-f}
is true modulo $\Phi_n(q)\Phi_n(-q)$ by taking $a,b\to 1$ and
$c\to -1$ in \eqref{eq:qDixon-1}. Moreover, it is also true when
$q=1$ and $n=p$ is an odd prime, since Tauraso observed that
\begin{equation}\label{eq:tauraso}
\sum_{k=0}^{n} \frac{4k+1}{16^k}{\binom{2k}k}^2
=\frac{(2n+1)^2}{16^n}{\binom{2n}n}^2.
\end{equation}
While a $q$-analogue of \eqref{eq:tauraso} was given by the first author
\cite{Guo2018}, namely
\begin{equation*}
\sum_{k=0}^nq^{-k}\begin{bmatrix}2k\\k\end{bmatrix}^2
(-q^{k+1};q)_{n-k}^4=q^{-n}[2n+1]^2
\begin{bmatrix}2n\\n\end{bmatrix}^2,
\end{equation*}
there is no closed form of the left-hand side of \eqref{4k+1-Dixon-f}.

\section{Some $q$-congruences from a double series transformation
of Ismail, Rahman and Suslov}\label{sec:irs}

In \cite[Thm.~1.1]{IRS} Ismail, Rahman and Suslov derived the
following transformation formula:
\begin{align}\label{eq:irs}
\sum_{k=0}^\infty&\frac{(1-aq^{2k})(a,b,c,d,e,f;q)_k}
{(1-a)(q,aq/b,aq/c,aq/d,aq/e,aq/f;q)_k}
\left(\frac{a^2q^2}{bcdef}\right)^k
{}_{4}\phi_{3}\!\left[\begin{array}{c}
q^{-k},\ aq^k,\ g,\ h \\
b,\, c,\, aghq/bc
\end{array};q,\, q
\right]\notag\\
&=\frac{(aq,aq/de,aq/df,aq/ef;q)_\infty}{(aq/d,aq/e,aq/f,aq/def;q)_\infty}
{}_{5}\phi_{4}\!\left[\begin{array}{c}
agq/bc,\ ahq/bc,\ d,\ e,\ f \\
aghq/bc,\ aq/b,\, aq/c,\, def/a
\end{array};q,\, q
\right]\notag\\
&\quad\;+{}
\frac{(aq,d,e,f,a^2q^2/bdef,a^2q^2/cdef,agq/bc,ahq/bc,a^2ghq^2/bcdef;q)_\infty}
{(aq/b,aq/c,aq/d,aq/e,aq/f,def/aq,aghq/bc,a^2gq^2/bcdef,a^2kq^2/bcdef;q)_\infty}
\notag\\&\qquad\;\times {}_{5}\phi_{4}\!\left[\begin{array}{c}
agq/de,\ aq/df,\ aq/ef,\ a^2gq^2/bcdef,\ a^2hq^2/bcdef \\
a^2q^2/bdef,\ a^2q^2/cdef,\, aq^2/def,\, a^2ghq^2/bcdef
\end{array};q,\, q
\right],
\end{align}
provided $|a^2q^2/bcdef|<1$.

If in \eqref{eq:irs} we replace $q$ by $q^3$, take $a=g=q$, $h=aq$
and $b=c=d=e=f=q^2$, and suitably truncate the sum, then the
following ``divergent'' $q$-supercongruence appears to be true.
\begin{conjecture}\label{conj:irs-1}
Let $n>1$ be a positive integer with $n\equiv 1\pmod 3$.Then
\begin{equation*}
\sum_{k=0}^{(2n-2)/3}[6k+1]\frac{(q;q^3)_k}{(q^3;q^3)_k}q^{-2k}
{}_{4}\phi_{3}\!\left[\begin{array}{c}
q^{-3k},\, q^{3k+1},\, q,\, aq \\
q^2,\, q^2,\, aq^2
\end{array};q^3,\, q^3
\right] \equiv 0 \pmod{\Phi_{n}(q)^2}.
\end{equation*}
Furthermore, the above congruence holds modulo $\Phi_n(q)^3$ when
$a=1$.
\end{conjecture}

On the other hand, if in \eqref{eq:irs} we replace $q$ by $q^3$,
take $a=g=q^{-1}$, $h=aq^{-1}$ and $b=c=d=e=f=q$, and suitably
truncate the sum, then the following ``divergent''
$q$-supercongruence appears to be true.
\begin{conjecture}Let $n>2$ be a positive integer with $n\equiv 2\pmod 3$. Then
\begin{equation*}
\sum_{k=0}^{(2n-1)/3}[6k-1]\frac{(q^{-1};q^3)_k}{(q^3;q^3)_k}q^{-k}
{}_{4}\phi_{3}\!\left[\begin{array}{c}
q^{-3k},\, q^{3k-1},\, q^{-1},\, aq^{-1} \\
q,\, q,\, aq^{-2}
\end{array};q^3,\, q^3
\right] \equiv 0 \pmod{\Phi_{n}(q)^2}.
\end{equation*}
Furthermore, the above congruence holds modulo $\Phi_n(q)^3$ when
$a=1$.
\end{conjecture}

If in \eqref{eq:irs} we replace $q$ by $q^4$, take $a=b=c=d=e=f=q$,
$g=q^{-1}$, $h=aq^{-1}$, and suitably truncate the sum, then the
following $q$-supercongruence appears to be true.
\begin{conjecture} Let $n$ be a positive integer with $n\equiv 3\pmod 4$. Then
\begin{align*}
\sum_{k=0}^{(3n-1)/4}[8k+1]\frac{(q;q^4)_k^6}{(q^4;q^4)_k^6}q^{5k}
{}_{4}\phi_{3}\!\left[\begin{array}{c}
q^{-4k},\, q^{4k+1},\, q^{-1},\, aq^{-1} \\
q,\, q,\, aq
\end{array};q^4,\, q^4
\right]\equiv 0 \pmod{[n]\Phi_{n}(q)^2}.
\end{align*}
Furthermore, the above congruence holds modulo $[n]\Phi_n(q)^3$ when
$a=1$.
\end{conjecture}

On the other hand, if in \eqref{eq:irs} we replace $q$ by $q^4$,
take $a=b=c=d=e=f=q^{-1}$, $g=q^{-3}$, $h=aq^{-3}$, and suitably
truncate the sum, then the following $q$-supercongruence appears to
be true.
\begin{conjecture} Let $n>1$ be a positive integer with $n\equiv 1\pmod 4$. Then
\begin{align*}
\sum_{k=0}^{(3n+1)/4}[8k-1]\frac{(q^{-1};q^4)_k^6}{(q^4;q^4)_k^6}q^{11k}
{}_{4}\phi_{3}\!\left[\begin{array}{c}
q^{-4k},\, q^{4k-1},\, q^{-3},\, aq^{-3} \\
q^{-1},\, q^{-1},\, aq^{-1}
\end{array};q^4,\, q^4
\right] \equiv 0 \pmod{[n]\Phi_{n}(q)^2}.
\end{align*}
Furthermore, the above congruence holds modulo $[n]\Phi_n(q)^3$ when
$a=1$.
\end{conjecture}

Ismail, Rahman and Suslov \cite[Eq.~(5.4)]{IRS} also noted the following
transformation formula (which can be obtained from \eqref{eq:irs} by taking
$d=aq/c$ and $h=0$):
\begin{align}
&\sum_{k=0}^\infty
\frac{(1-aq^{2k})(a,b,e,f;q)_k}
{(1-a)(q,aq/b,aq/e,aq/f;q)_k
} \left(\frac{aq}{bef}\right)^k
{}_{3}\phi_{2}\!\left[\begin{array}{c}
q^{-k},\, aq^{k},\, g \\
b,\, c
\end{array};q,\, q
\right]                \notag\\[5pt]
&\quad=\frac{(aq,c/e,c/f,aq/ef;q)_\infty}{(c,aq/e,aq/f,c/ef;q)_\infty}
{}_{3}\phi_{2}\!\left[\begin{array}{c}
agq/bc,\, e,\, f \\
efq/c,\, aq/b
\end{array};q,\, q
\right]  \notag\\[5pt]
&\qquad\quad{}+\frac{(aq,e,f,acq/bef,aq/ef,agq/bc;q)_\infty}
{(c,aq/b,aq/e,aq/f,ef/c,aqg/bef;q)_\infty}
{}_{3}\phi_{2}\!\left[\begin{array}{c}
c/e,\, c/f,\, agq/bef \\
acq/bef,\,cq/ef
\end{array};q,\, q
\right].  \label{eq:irs5.4}
\end{align}
If in \eqref{eq:irs5.4} we replace $q$ by $q^4$, take
$a=b=c=e=f=q^{-2}$, $g=q^5$ and truncate the sum, then the following
$q$-supercongruence appears to be true.
\begin{conjecture}\label{conj:irs-5}
Let $n$ be a positive integer with $n\equiv 3\pmod{8}$. Then
\begin{align}
\sum_{k=0}^{(n+1)/2}[8k-2]\frac{(q^{-2};q^4)_k^4}{(q^4;q^4)_k^4}
q^{8k} {}_{3}\phi_{2}\!\left[\begin{array}{c}
q^{-4k},\, q^{4k-2},\,q^5 \\
q^{-2},\,q^{-2}
\end{array};q^4,\, q^4
\right] \equiv 0\pmod{\Phi_n(q)^2 \Phi_n(-q)}.  \label{eq:irs-5}
\end{align}
\end{conjecture}

Similarly as before, we can show that all the congruences in
Conjectures \ref{conj:irs-1}--\ref{conj:irs-5} are true modulo
$\Phi_n(q)$. For example, we have the following parametric generalization
of the congruence \eqref{eq:irs-5} modulo $\Phi_n(q)\Phi_n(-q)$.
\begin{theorem}\label{conj:irs-5a}
Let $n$ be a positive integer with $n\equiv 3\pmod{8}$. Then, modulo
$\Phi_n(q) \Phi_n(-q)$,
\begin{align*}
\sum_{k=0}^{(n+1)/2}[8k-2]\frac{(q^{-2};q^4)_k^2
(aq^{-2},q^{-2}/a;q^4)_k}{(q^4;q^4)_k^2 (aq^{4},q^4/a;q^4)_k} q^{8k}
{}_{3}\phi_{2}\!\left[\begin{array}{c}
q^{-4k},\, q^{4k-2},\,q^5 \\
q^{-2},\,q^{-2}
\end{array};q^4,\, q^4
\right] \equiv 0.
\end{align*}
\end{theorem}
\begin{proof}Let $q\to q^4$, $a=q^{-2-2n}$, $b=c=q^{-2}$, $e=aq^{-2}$,
and $f=q^{-2}/a$ in \eqref{eq:irs5.4}. Then the left-hand side terminates
at $k=(n+1)/2$ because of the factor $(q^{-2-2n};q^4)_k$ in the numerator,
while the right-hand side vanishes because of the factor $(q^{2-2n};q^4)_\infty$.
The described specialization thus yields the following identity:

\begin{align*}
&\sum_{k=0}^{(n+1)/2}(1-q^{8k-2-2n})\frac{(q^{-2-2n},q^{-2},aq^{-2},q^{-2}/a;q^4)_k}
{(q^4,q^{4-2n},aq^{4-2n},q^{4-2n}/a;q^4)_k} q^{(8-2n)k} \\[5pt]
&\quad\qquad\times {}_{3}\phi_{2}\!\left[\begin{array}{c}
q^{-4k},\, q^{4k-2-2n},\,q^5 \\
q^{-2},\,q^{-2}
\end{array};q^4,\, q^4
\right] =0.
\end{align*}
Since $q^{2n}\equiv 1\pmod{\Phi_n(q)\Phi_n(-q)}$,
we immediately deduce the desired congruence from the above identity.
\end{proof}

\section{Concluding remarks and further open problems}\label{sec:concl}
Most of the congruences in the manuscript \cite{GuoZu} are modulo
$[n](1-aq^n)(a-q^n)$. However, the congruence \eqref{eq:qab} does
not hold modulo $[n](1-aq^n)(a-q^n)$ in general. We only have a
generalization of \eqref{eq:qab} with $a=1$.

It is easy to see that the following generalization of
\eqref{eq:first} in Theorem~\ref{thm:first} is true.
\begin{theorem}\label{conj:with-b}
Let $n$ be a positive odd integer. Then, modulo $[n]\Phi_n(q)^2$,
\begin{equation}
\sum_{k=0}^{(n-1)/2} [4k+1]\frac{(bq;q^2)_k (q;q^2)_k^5}
{(q^2/b;q^2)_k (q^2;q^2)_k^5}\left(\frac{q}{b}\right)^k \equiv
[n]q^{(1-n)/2} \sum_{k=0}^{(n-1)/2}\frac{(q/b;q^2)_k (q;q^2)_k^3 }
{(q^2/b;q^2)_k (q^2;q^2)_k^3}q^{2k}. \label{eq:conj-1}
\end{equation}
\end{theorem}
Letting $a=1$ in Theorem~\ref{thm:2.2}, we see that the congruence
\eqref{eq:conj-1} holds modulo $\Phi_n(q)^3$. Therefore,
Theorem~\ref{conj:with-b} is equivalent to the left-hand side of
\eqref{eq:conj-1} being congruent to $0$ modulo $[n]$. By
\eqref{eq:qab}, we see that the left-hand side of \eqref{eq:conj-1}
is congruent to $0$ modulo $\Phi_n(q)$. And the same technique to
prove congruences modulo $[n]$ from congruences modulo $\Phi_n(q)$
as used in the proofs of \eqref{eq:4k+1} and \eqref{eq:4k+1,b} still works here.



We conjecture that the following generalization of the second part
of Theorem~\ref{thm:third} is true.
\begin{conjecture}\label{conj:11.2}
Let $n$ be a positive integer with $n\equiv 2\pmod 3$. Then
\begin{equation*}
\sum_{k=0}^{n-1}[6k+1]\frac{(q;q^3)_k^6}{(q^3;q^3)_k^6} q^{3k}
\equiv 0 \pmod{[n]\Phi_n(q)^3}.
\end{equation*}
\end{conjecture}

We also have the following similar conjecture.
\begin{conjecture}\label{conj:11.3}
Let $n>1$ be a positive integer with $n\equiv 1\pmod 3$. Then
\begin{equation*}
\sum_{k=0}^{n-1}[6k-1]\frac{(q^{-1};q^3)_k^6}{(q^3;q^3)_k^6} q^{9k}
\equiv 0 \pmod{[n]\Phi_n(q)^3}.
\end{equation*}
\end{conjecture}

Note that, similar to the proof of Theorem~\ref{thm:third}, we can
show that the above congruence holds modulo $[n]\Phi_n(q)$. We point
out that $q$-congruences modulo $[n]\Phi_n(q)^3$ or $\Phi_n(q)^4$
are very difficult to prove. As far as we know, the following result
\begin{equation*}
\sum_{k=0}^{(n-1)/2}[4k+1]\frac{(q;q^2)_k^4}{(q^2;q^2)_k^4} \equiv
q^{(1-n)/2}[n]+\frac{(n^2-1)(1-q)^2}{24}\,q^{(1-n)/2}[n]^3
\pmod{[n]\Phi_n(q)^3},
\end{equation*}
due to the first author and Wang~\cite{GW}, is the unique
$q$-congruence modulo $[n]\Phi_n(q)^3$ in the literature that is
completely proved. (Several similar conjectural $q$-congruences are
collected in \cite{GuoZu}.) It is natural to ask whether there is a
complete $q$-analogue of Long's supercongruence \eqref{eq:long-1}.

Inspired by the $q$-congruences in the previous sections, we shall
propose the following conjecture.
\begin{conjecture}Let $n$ be a positive odd integer. Then
\begin{equation*}
\sum_{k=0}^{n-1} [8k+1]\frac{(q;q^4)_k^6
(q^2;q^2)_{2k}}{(q^4;q^4)_k^6 (q;q^2)_{2k}} q^{4k} \equiv
\begin{cases} 0  \pmod{[n]}, &\text{if $n\equiv 1\pmod 4$,}\\[10pt]
 0 \pmod{[n]\Phi_n(q)^2}, &\text{if $n\equiv 3\pmod 4$.}
\end{cases}
\end{equation*}
\end{conjecture}
Note that the left-hand side is not a truncated form of
\eqref{12phi11} with $q\mapsto q^4$ and $a=b=c=d=q$. Therefore, even
for the case modulo $\Phi_n(q)$, the above conjecture is still open.
Moreover, we cannot find any parametric generalization of the above
conjecture, although one would believe that such a generalization should
exist.

Similarly, the following conjecture seems to be true.
\begin{conjecture}Let $n>1$ be a positive odd integer. Then
\begin{equation*}
\sum_{k=0}^{n-1} [8k-1]\frac{(q^{-1};q^4)_k^6
(q^2;q^2)_{2k}}{(q^4;q^4)_k^6 (q^{-1};q^2)_{2k}} q^{8k} \equiv
\begin{cases} 0  \pmod{[n]\Phi_n(q)^2}, &\text{if $n\equiv 1\pmod 4$,}\\[10pt]
 0 \pmod{[n]}, &\text{if $n\equiv 3\pmod 4$.}
\end{cases}
\end{equation*}
\end{conjecture}

For the $q=1$ case of \eqref{eq:6k+1-2}, much more seems to be true.
Numerical computations suggest the following result.
\begin{conjecture}\label{conj:7.5}
Let $p\equiv 1\pmod{3}$. Then
\begin{equation*}
\sum_{k=0}^{(p-1)/3} (6k+1)\frac{(\frac{1}{3})_k^4 (2k)!} {k!^4
(\frac{2}{3})_{2k}}\equiv p\pmod{p^3}.
\end{equation*}
\end{conjecture}

We also have a similar conjecture related to \eqref{eq:6k-1-2}.
\begin{conjecture}\label{conj:7.6}
Let $p\equiv 2\pmod{3}$. Then
\begin{equation*}
\sum_{k=0}^{(p+1)/3} (6k-1)\frac{(-\frac{1}{3})_k^4 (2k)!} {k!^4
(-\frac{2}{3})_{2k}}\equiv p\pmod{p^3}.
\end{equation*}
\end{conjecture}

Unfortunately, we failed to find complete $q$-analogues of the above
two conjectures. In particular, we do not know how to use the
creative microscoping method to tackle them.

In \cite[Conj.~5.4]{Guo2018} the first author has made the
following conjecture.
\begin{conjecture}\label{conj:q-Hamme}
Let $n$ and $r$ be positive integers. Then
\begin{subequations}
\begin{align}
\sum_{k=0}^n
(-1)^{k}q^{k^2+(r-2)k}[4k+1]\begin{bmatrix}2k\\k\end{bmatrix}^{2r-1}
(-q^{k+1};q)_{n-k}^{4r-2} \equiv 0 \pmod{
(1+q^{n})^{2r-2}[2n+1]\begin{bmatrix}2n\\n\end{bmatrix}},
\label{eq:sub-a} \\
\intertext{and} \sum_{k=0}^n
q^{(r-2)k}[4k+1]\begin{bmatrix}2k\\k\end{bmatrix}^{2r}
(-q^{k+1};q)_{n-k}^{4r} \equiv 0 \pmod{
(1+q^{n})^{2r-1}[2n+1]\begin{bmatrix}2k\\k\end{bmatrix}}.
\label{eq:sub-b}
\end{align}
\end{subequations}
\end{conjecture}
Note that the congruences \eqref{eq:sub-a} for $r=1,2$ and
\eqref{eq:sub-b} for $r=1$ have been proved by the first author
\cite{Guo2018} himself, and the congruence \eqref{eq:sub-b} for
$r=2$ has been established by the first author and Wang \cite{GW}.

In this section, we shall prove the following weaker form of the
above conjecture.
\begin{theorem}The congruences \eqref{eq:sub-a} and \eqref{eq:sub-b}
are true modulo $[n+1][2n+1]$.
\end{theorem}
\begin{proof}
The proof is similar to the second step of that of
Theorem~\ref{thm:2.2}. By Lemma~\ref{lem:2.1}, it is easy to see
that for odd $n$ we have
\begin{align*}
&(-1)^{(n-1)/2-k}q^{((n-1)/2-k)^2+(r-2)((n-1)/2-k)}
\frac{[2n-4k-1](q;q^2)_{(n-1)/2-k}^{2r-1}}
{(q^2;q^2)_{(n-1)/2-k}^{2r-1}}  \notag\\[5pt]
&\quad\equiv -(-1)^k q^{k^2+(r-2)k}[4k+1]\frac{(q;q^2)_k^{2r-1}}
{(q^2;q^2)_k^{2r-1}} \pmod{\Phi_n(q)},
\end{align*}
and
\begin{align*}
&q^{(r-2)((n-1)/2-k)}\frac{[2n-4k-1](q;q^2)_{(n-1)/2-k}^{2r}}
{(q^2;q^2)_{(n-1)/2-k}^{2r}}  \notag\\[5pt]
&\quad\equiv -q^{(r-2)k}[4k+1]\frac{(q;q^2)_k^{2r}}
{(q^2;q^2)_k^{2r}} \pmod{\Phi_n(q)}.
\end{align*}
This proves that, for odd $n$,
\begin{equation*}
\sum_{k=0}^{(n-1)/2}(-1)^k
q^{k^2+(r-2)k}[4k+1]\frac{(q;q^2)_k^{2r-1}}{(q^2;q^2)_k^{2r-1}}
\equiv 0\pmod{\Phi_n(q)},
\end{equation*}
and
\begin{equation*}
\sum_{k=0}^{(n-1)/2}q^{(r-2)k}[4k+1]\frac{(q;q^2)_k^{2r}}{(q^2;q^2)_k^{2r}}
\equiv 0\pmod{\Phi_n(q)}.
\end{equation*}

Similar to the proofs of \eqref{eq:4k+1} and \eqref{eq:4k+1,b}, we
can further show that
\begin{equation*}
\sum_{k=0}^{M}(-1)^k
q^{k^2+(r-2)k}[4k+1]\frac{(q;q^2)_k^{2r-1}}{(q^2;q^2)_k^{2r-1}}
\equiv 0\pmod{[n]},
\end{equation*}
and
\begin{equation*}
\sum_{k=0}^{M}q^{(r-2)k}[4k+1]\frac{(q;q^2)_k^{2r}}{(q^2;q^2)_k^{2r}}
\equiv 0\pmod{[n]},
\end{equation*}
where $M=n-1$ or $(n-1)/2$. It is easy to see that the polynomials
$[n+1]$ and $[2n+1]$ are relatively prime. The proof then follows
from the above two congruences by replacing $n$ with $n+1$ or $2n+1$
and using the relation \eqref{eq:q-bino-1/2}.
\end{proof}

Finally, we consider the general very-well-poised
${}_{2d}\phi_{2d-1}$ series (which satisfies Slater's transformation
\cite[$r=d$ and $b_{2r}=a$ in Eq.~(5.5.2)]{GR}) where we replace $q$
by $q^d$ and take all upper parameters to be $q$. Then the following
further generalization of Conjecture \ref{conj:11.2} appears to be true.

\begin{conjecture}\label{conj:gs-1}
Let $d\geqslant 3$ and $n$ be positive integers
with $n\equiv -1\pmod{d}$. Then
\begin{equation*}
\sum_{k=0}^{M}[2dk+1]
\frac{(q;q^d)_k^{2d}}{(q^d;q^d)_k^{2d}}q^{d(d-2)k} \equiv 0
\pmod{[n]\Phi_{n}(q)^3},
\end{equation*}
where $M=((d-1)n-1)/d$ or $n-1$.
\end{conjecture}

Similarly, we consider the general very-well-poised
${}_{2d}\phi_{2d-1}$ series where we replace $q$ by $q^d$ and take
all upper parameters to be $q^{-1}$. Then the following
generalization of Conjecture \ref{conj:11.3} appears to be true.

\begin{conjecture}\label{conj:gs-2}
Let $d\geqslant 3$ and $n>1$ be positive integers
with $n\equiv 1\pmod{d}$. Then
\begin{equation*}
\sum_{k=0}^{M}[2dk-1]
\frac{(q^{-1};q^d)_k^{2d}}{(q^d;q^d)_k^{2d}}q^{d^2 k} \equiv 0
\pmod{[n]\Phi_{n}(q)^3},
\end{equation*}
where $M=((d-1)n+1)/d$ or $n-1$.
\end{conjecture}

\noindent{\it Remark.}
Since the submission of the original version of
this paper (which also appeared as a preprint on
the arXiv) and the present final version, relevant
developments have taken place. In particular, meanwhile
some of our conjectures were confirmed by other authors.
Recently, Wang~\cite{CWang} proved
Conjecture~\ref{conj:3k-1} and he further extended it to the modulus $p^4$ case.
The first author and Zudilin~\cite{GuoZu2} proved Conjecture~\ref{conj:GuoZu},
Jana and Kalita~\cite{JK} proved
Conjectures \ref{conj:jk-1}, \ref{conj:7.5} and \ref{conj:7.6},
Ni and Pan~\cite{NP} proved Conjecture~\ref{conj:q-Hamme},
and the authors themselves~\cite{GS2,GS4} confirmed
Conjectures~\ref{conj:4k-1}, \ref{conj:gs-1} and \ref{conj:gs-2}.
Moreover, Liu~\cite{Liu2} generalized Conjectures~\ref{conj:7.5}
and \ref{conj:7.6} to the modulus $p^4$ case.

\begin{appendix}
\section{A new nonterminating ${}_{12}\phi_{11}$
transformation and linearization of the continuous
$q$-ultraspherical polynomials}\label{sec:newtr}
Our starting point for deriving the nonterminating transformation formula in
Theorem~\ref{newtf} is the following transformation formula between two
terminating very-well-poised ${}_{14}\phi_{13}$ series from
\cite[Thm.~4.1]{LSW}.
\begin{align}\label{14phi13}
&\sum_{k=0}^n\frac{1-aq^{2k}}{1-a}
\frac{(a,b,c,d,ab/c,ab/d,abq^n,q^{-n};q)_k}
{(q,aq/b,aq/c,aq/d,cq/b,dq/b,q^{1-n}/b,aq^{n+1};q)_k}
\frac{(aq/b;q)_{2k}}{(ab;q)_{2k}}\Big(\frac qb\Big)^{2k}\notag\\
&=\frac{(aq,\hat aq/c,\hat aq/d,aq/cd;q)_n}
{(\hat aq,aq/c,aq/d,\hat aq/cd;q)_n}\notag\allowdisplaybreaks[0]\\
&\quad\;\times\sum_{k=0}^n\frac{1-\hat aq^{2k}}{1-\hat a}
\frac{(\hat a,b,c,d,\hat ab/c,\hat ab/d,\hat abq^n,q^{-n};q)_k}
{(q,\hat aq/b,\hat aq/c,\hat aq/d,cq/b,dq/b,q^{1-n}/b,\hat aq^{n+1};q)_k}
\frac{(\hat aq/b;q)_{2k}}{(\hat ab;q)_{2k}}\Big(\frac qb\Big)^{2k},
\end{align}
where $\hat a=q^{-n}cd/ab$.

We now have the following new transformation for a nonterminating very-well-poised
${}_{12}\phi_{11}$ series into two multiples of nonterminating
${}_4\phi_3$ series.
\begin{theorem}\label{newtf}
We have
\begin{align}\label{12phi11}
&\sum_{k=0}^\infty\frac{1-aq^{2k}}{1-a}
\frac{(a,b,c,d,ab/c,ab/d;q)_k}
{(q,aq/b,aq/c,aq/d,cq/b,dq/b;q)_k}
\frac{(aq/b;q)_{2k}}{(ab;q)_{2k}}\Big(\frac qb\Big)^{k}\notag\\
&=\frac{(aq,ab/c,ab/d,aq/cd;q)_\infty}
{(ab,aq/c,aq/d,ab/cd;q)_\infty}\sum_{k=0}^\infty
\frac{(b,c,d,cd/a;q)_k}
{(q,cq/b,dq/b,cdq/ab;q)_k}\Big(\frac qb\Big)^{2k}\notag\allowdisplaybreaks[0]\\
&\quad\;+
\frac{(aq,c,d,cdq/ab^2;q)_\infty}
{(ab,cq/b,dq/b,cd/ab;q)_\infty}\sum_{k=0}^\infty
\frac{(b,ab/c,ab/d,ab^2/cd;q)_k}
{(q,aq/c,aq/d,abq/cd;q)_k}\Big(\frac qb\Big)^{2k},
\end{align}
where $|q/b|<1$.
\end{theorem}
This result extends Gasper's \cite[Eq.~(3.2)]{G} (see also
\cite[Ex.~8.15]{GR}). Observe that the two ${}_4\phi_3$ series on
the right-hand side are not balanced, nor well-poised. However, they
satisfy the remarkable property that the quotient (not the product!)
of corresponding upper and lower parameters is throughout the same,
namely $b/q$.

By replacing $a$, $b$, $c$, $d$ in \eqref{12phi11} by
$q^a$, $q^b$, $q^c$, $q^d$, respectively, and letting $q\to 1^-$
we obtain the following transformation between a
nonterminating very-well-poised ${}_9F_8$ series into two multiples of
nonterminating ${}_4F_3$ series.
(For the notion of a hypergeometric ${}_rF_s$ series, see \cite{AAR}.
In the following, we employ the condensed notation for products of
Pochhammer symbols, $(a_1,\dots,a_m)_k=(a_1)_k\cdots(a_m)_k$.)
\begin{align}\label{9F8}
&\sum_{k=0}^\infty\frac{a/2+k}{a/2}
\frac{(a,b,c,d,a+b-c,a+b-d)_k}
{(1,a+1-b,a+1-c,a+1-d,c+1-b,d+1-b)_k}
\frac{(a+1-b)_{2k}}{(a+b)_{2k}}\notag\\
&=\frac{\Gamma(a+b)\Gamma(a+1-c)\Gamma(a+1-d)\Gamma(a+b-c-d)}
{\Gamma(a+1)\Gamma(a+b-c)\Gamma(a+b-d)\Gamma(a+1-c-d)}
\notag\allowdisplaybreaks[0]\\
&\quad\;\times
\sum_{k=0}^\infty
\frac{(b,c,d,c+d-a)_k}
{(1,c+1-b,d+1-b,c+d+1-a-b)_k}\notag\allowdisplaybreaks[0]\\
&\quad\;+
\frac{\Gamma(a+b)\Gamma(c+1-b)\Gamma(d+1-b)\Gamma(c+d-a-b)}
{\Gamma(a+1)\Gamma(c)\Gamma(d)\Gamma(c+d+1-a-2b)}
\notag\allowdisplaybreaks[0]\\
&\qquad\;\times\sum_{k=0}^\infty
\frac{(b,a+b-c,a+b-d,a+2b-c-d)_k}
{(1,a+1-c,a+1-d,a+b+1-c-d)_k},
\end{align}
where, for convergence, $\Re(b)<\frac 34$.

The transformation in \eqref{9F8} extends \cite[Eq.~(3.3)]{G}.
\begin{proof}[Proof of Theorem~\ref{newtf}]
We would like to take $n\to\infty$ in \eqref{14phi13} but the series
on the right-hand side has large terms near the end compared to
those in the middle of the series which prevents us from taking
the term-by-term limit directly.
We thus apply a similar analysis as applied by
Bailey~\cite[Eq.~8.5(3)]{B} in his derivation
of the nonterminating Watson transformation (who started with
the terminating balanced very-well-poised ${}_{10}\phi_9$ transformation
to derive a transformation of a nonterminating very-well-poised
${}_8\phi_7$ series into two multiples of balanced ${}_4\phi_3$ series),
see also \cite[Sec.~2.10]{GR}.
In \eqref{14phi13}, we first replace $n$ by $2m+1$. Then we write the
series on the right-hand side as
\begin{equation}
\sum_{k=0}^{2m+1}\lambda_k=\sum_{k=0}^m\lambda_k+
\sum_{k=m+1}^{2m+1}\lambda_k=\sum_{k=0}^m\lambda_k+
\sum_{k=0}^m\lambda_{2m+1-k},
\end{equation}
after which in each of the sums we can term-wise let $m\to\infty$
(which is justified by Tannery's theorem~\cite[p.~292]{Tannery}
under the restriction $|q/b|<1$). The identity in \eqref{12phi11} thus follows.
\end{proof}

Notice that if in \eqref{12phi11} we take $d=a/c$ the first
series on the right-hand side reduces to $1$.
(If instead $d=ab^2/c$ then the second series on the right-hand
side reduces to $1$. The resulting series is equivalent
to \eqref{corid} by the substitution $c\mapsto c/b$.)
We thus have the following nonterminating very-well-poised
${}_{12}\phi_{11}$ summation:
\begin{corollary}\label{cor:A2}
We have
\begin{align}\label{corid}
&\sum_{k=0}^\infty\frac{(1-aq^{2k})(1-c)(1-a/c)}{(1-a)(1-aq^k/c)(1-cq^k)}
\frac{(a,b,bc,ab/c;q)_k}
{(q,aq/b,aq/bc,cq/b;q)_k}
\frac{(aq/b;q)_{2k}}{(ab;q)_{2k}}\Big(\frac qb\Big)^{k}\notag\\
&=\frac{(q,aq,bc,ab/c;q)_\infty}
{(b,cq,aq/c,ab;q)_\infty}
+\frac{(aq,c,a/c,q/b^2;q)_\infty}
{(ab,cq/b,aq/bc,1/b;q)_\infty}\sum_{k=0}^\infty
\frac{(b,b^2,bc,ab/c;q)_k}
{(q,bq,cq,aq/c;q)_k}\Big(\frac qb\Big)^{2k},
\end{align}
where $|q/b|<1$.
\end{corollary}

A further special case of Corollary~\ref{cor:A2} is worth stating
separately. If we take $b=-1$, the second sum reduces to its first term.
We then obtain the following nonterminating very-well-poised
${}_{12}\phi_{11}$ summation where the two terms on the right-hand side
satisfy a nice symmetry.
\begin{corollary}\label{cor:A3}
We have
\begin{align}\label{corid3}
&\sum_{k=0}^\infty\frac{(1-a^2q^{4k})(1-c^2)(1-a^2/c^2)}
{(1-a^2)(1-c^2q^{2k})(1-a^2q^{2k}/c^2)}
\frac{(a,-1;q)_k}
{(q,-aq;q)_k}(-q)^{k}\notag\\
&=\frac{(q,aq;q)_\infty}
{2\,(-q,-a;q)_\infty}\left(
\frac{(-c,-a/c;q)_\infty}
{(cq,aq/c;q)_\infty}
+\frac{(c,a/c;q)_\infty}
{(-cq,-aq/c;q)_\infty}\right).
\end{align}
\end{corollary}

If in \eqref{12phi11} we take (instead of $d=a/c$ which led to
Corollary~\ref{cor:A2}) $d=aq/c$ the prefactor of the first
series on the right-hand vanishes.
(If instead $d=ab^2/cq$ then the prefactor of the second series
on the right-hand vanishes. The resulting series is equivalent
to \eqref{coridb} by the substitution $c\mapsto ab/c$.)
We thus have the following nonterminating very-well-poised
${}_{10}\phi_9$ transformation:
\begin{corollary}\label{cor:A2b}
We have
\begin{align}\label{coridb}
&\sum_{k=0}^\infty\frac{1-aq^{2k}}{1-a}
\frac{(a,b,ab/c,bc/q;q)_k}
{(q,aq/b,cq/b,aq^2/bc;q)_k}
\frac{(aq/b;q)_{2k}}{(ab;q)_{2k}}\Big(\frac qb\Big)^{k}\notag\\
&=\frac{(aq,q^2/b^2,c,aq/c;q)_\infty}
{(ab,cq/b,aq^2/bc,q/b;q)_\infty}\sum_{k=0}^\infty
\frac{(b^2/q,bc/q,ab/c;q)_k}
{(q,aq/c,c;q)_k}\Big(\frac qb\Big)^{2k},
\end{align}
where $|q/b|<1$.
\end{corollary}

We also record another (simpler) special case of \eqref{14phi13},
obtained by taking $d\to\infty$. Alternatively, it can be obtained from
Theorem~\ref{newtf} by choosing $d=q^{-n}$.
\begin{corollary}\label{cortf}
We have
\begin{align}\label{12phi11t}
\sum_{k=0}^n\frac{1-aq^{2k}}{1-a}
\frac{(a,b,c,ab/c,abq^n,q^{-n};q)_k}
{(q,aq/b,aq/c,cq/b,q^{1-n}/b,aq^{n+1};q)_k}
\frac{(aq/b;q)_{2k}}{(ab;q)_{2k}}\Big(\frac qb\Big)^{k}&\notag\\
=\frac{(aq,ab/c;q)_n}
{(ab,aq/c;q)_n}\sum_{k=0}^n
\frac{(b,c,q^{-n}c/a,q^{-n};q)_k}
{(q,cq/b,cq^{1-n}/ab,q^{1-n}/b;q)_k}\Big(\frac qb\Big)^{2k}&.
\end{align}
\end{corollary}

We use Corollary~\ref{cortf}, which is equivalent to
Gasper's aforementioned identity in \cite[Eq.~(3.2)]{G} (also
stated in \cite[Ex.~8.15]{GR}),  to provide a generalization of
Rogers' linearization formula for the continuous
$q$-ultraspherical polynomials in \eqref{linear}.

The continuous $q$-ultraspherical polynomials, which depend on
a parameter $\ba$ and the base $q$, are given by
\begin{equation}\label{defC}
C_n(x;\ba\,|\,q)=\sum_{k=0}^n\frac{(\ba;q)_k\,(\ba;q)_{n-k}}
{(q;q)_k\,(q;q)_{n-k}}e^{i(n-2k)\ta}, \qquad x=\cos\ta.
\end{equation}
(Note that $\ta$ need not be real.)
They were originally considered by Rogers~\cite{R3} in 1884
(not aware of their orthogonality) in the pursuit of (what is now called)
the Rogers--Ramanujan identities.

These functions, which can be written as
\begin{equation}
C_n(x;\ba\,|\,q)=\frac{(\ba;q)_n}{(q;q)_n}e^{in\ta}\,
{}_2\phi_1\!\left[\begin{matrix}
\ba,q^{-n}\\q^{1-n}/\ba\end{matrix};q,\frac{qe^{-2i\ta}}{\ba}\right],
\end{equation}
are polynomials in $x$ of degree $n$.

As was established by Askey and Ismail~\cite{AI} (see also
\cite[Thm.~13.2.1]{Ibook}), the continuous $q$-ultraspherical polynomials
satisfy for $|\ba|<1$ the orthogonality relation
\begin{align}
\frac 1{2\pi}\int_{-1}^1C_m(x;\ba\,|\,q)\,C_n(x;\ba\,|\,q)\,
\frac{(e^{\pm 2i\ta};q)_\infty}{(\ba e^{\pm 2i\ta};q)_\infty}
\frac{\rd x}{\sqrt{1-x^2}}&\notag\\
\label{orth}=
\frac{(\ba,q\ba;q)_\infty}{(q,\ba^2;q)_\infty}
\frac{(\ba^2;q)_n}{(q;q)_n}\frac{(1-\ba)}{(1-\ba q^n)}\,
\da_{m,n}&.
\end{align}

Rogers~\cite{R3} (cf.\ also \cite[Thm.~13.3.2]{Ibook})
derived the following linearization formula:
\begin{align}\label{linear}
&C_m(x;\ba\,|\,q)\,C_n(x;\ba\,|\,q)\notag\\
&=\sum_{k=0}^{\min(m,n)}
\frac{(q;q)_{m+n-2k}(\ba;q)_{m-k}(\ba;q)_{n-k}(\ba;q)_k(\ba^2;q)_{m+n-k}}
{(\ba^2;q)_{m+n-2k}(q;q)_{m-k}(q;q)_{n-k}(q;q)_k(q\ba;q)_{m+n-k}}
\notag\allowdisplaybreaks[0]\\
&\qquad\qquad\;\times\frac{1-\ba q^{m+n-2k}}{1-\ba}\,C_{m+n-2k}(x;\ba\,|\,q).
\end{align}

Rogers' proof of \eqref{linear} involved induction. Other proofs have been
given by Bressoud~\cite{Br}, Rahman~\cite{R}, and by Gasper~\cite{G}
(see also \cite[Sec.~8.5]{GR}).
For our extension of Rogers' linearization formula we define, for
$\beta$ and $\nu$ being any two complex numbers,
\begin{equation}\label{defF}
F_\nu(z;\ba\,|\,q)=\sum_{k\geqslant
0}\frac{(\ba;q)_k\,(\ba;q)_{\nu-k}} {(q;q)_k\,(q;q)_{\nu-k}}z^{2k},
\end{equation}
which is an even analytic function in $z$. (Recall that in \eqref{qpoch}
the $q$-shifted factorials were defined for any complex subindex.)
These functions can be written as
\begin{equation}
F_\nu(z;\ba\,|\,q)=\frac{(q^{1+\nu},\ba;q)_\infty}{(q,\ba q^\nu;q)_\infty}\,
{}_2\phi_1\!\left[\begin{matrix}
\ba,q^{-\nu}\\q^{1-\nu}/\ba\end{matrix};q,\frac{qz^2}{\ba}\right],
\end{equation}
where $|qz^2/\ba|<1$, for absolute convergence.
For $\nu=n$ being a nonnegative integer,
we have $F_n(z;\ba\,|\,q)=z^nC_n(x;\ba\,|\,q)$,
where $x=(z+z^{-1})/2$ (or, equivalently, $x=\cos\ta$ with $z=e^{i\ta}$).

\begin{theorem}\label{thm:linearF}
Let $\mu,\nu$ and $\ba$ be three complex numbers. Then we have
the following identity of power series in $z$.
\begin{align}\label{linearF}
&F_\mu(z;\ba\,|\,q)\,F_\nu(z;\ba\,|\,q)\notag\\
&=\sum_{k\geqslant 0}
\frac{(q;q)_{\mu+\nu-2k}(\ba;q)_{\mu-k}(\ba;q)_{\nu-k}(\ba;q)_k
(\ba^2;q)_{\mu+\nu-k}}
{(\ba^2;q)_{\mu+\nu-2k}(q;q)_{\mu-k}(q;q)_{\nu-k}(q;q)_k
(q\ba;q)_{\mu+\nu-k}}\notag\allowdisplaybreaks[0]\\
&\qquad\quad\;\times\frac{1-\ba q^{\mu+\nu-2k}}{1-\ba}\,z^{2k}\,
F_{\mu+\nu-2k}(z;\ba\,|\,q),
\end{align}
where $|qz^2/\beta|<1$, for absolute convergence.
\end{theorem}

\begin{proof}
In the subsequent manipulations of series we silently interchange double sums
which is justified as all the involved series absolutely converge (if they do
not terminate) for $|qz^2/\ba|<1$. We start with expanding the product of the
two functions on the left-hand side of \eqref{linearF} using \eqref{defF}:
\begin{align*}
&F_\mu(z;\ba,|\,q)\,F_\nu(z;\ba\,|\,q)\\
&=\sum_{k\geqslant 0}
\frac{(\ba;q)_k(\ba;q)_{\mu-k}}{(q;q)_k(q;q)_{\mu-k}}z^{2k}
\sum_{l\geqslant 0}\frac{(\ba;q)_l(\ba;q)_{\nu-l}}
{(q;q)_l(q;q)_{\nu-l}}z^{2l}\\
&\underset{k\mapsto k-l}{=}\sum_{l\geqslant 0}\sum_{k\geqslant l}
\frac{(\ba;q)_{k-l}(\ba;q)_{\mu-k+l}}{(q;q)_{k-l}(q;q)_{\mu-k+l}}z^{2k-2l}
\frac{(\ba;q)_l(\ba;q)_{\nu-l}}
{(q;q)_l(q;q)_{\nu-l}}z^{2l}\\
&=\sum_{k\geqslant 0}\sum_{0\le l\le k}
\frac{(\ba;q)_{k-l}(\ba;q)_{\mu-k+l}}{(q;q)_{k-l}(q;q)_{\mu-k+l}}z^{2k-2l}
\frac{(\ba;q)_l(\ba;q)_{\nu-l}}
{(q;q)_l(q;q)_{\nu-l}}z^{2l}\\
&=\sum_{k\geqslant 0}
\frac{(\ba;q)_k(\ba;q)_{\mu-k}}{(q;q)_k(q;q)_{\mu-k}}
\frac{(\ba;q)_\nu}{(q;q)_\nu}z^{2k}\allowdisplaybreaks[0]\\
&\hphantom{\underset{k\mapsto k-l}{=}}\times
\sum_{0\le l\le k}\frac{(\ba,q^{-k},q^{\mu-k}\ba,q^{-\nu};q)_l}
{(q,q^{1-k}/\ba,q^{1+\mu-k},q^{1-\nu}/\ba;q)_l}\Big(\frac q\beta\Big)^{2l}.
\end{align*}
Now the inner sum over $l$, which is a special terminating ${}_4\phi_3$ series,
can be transformed into a multiple of a ${}_{12}\phi_{11}$ series
using the $(a,b,c,n)\mapsto(\ba q^{\mu+\nu-2k},\ba,q^{\mu-k}\ba,k)$
case of \eqref{12phi11t}. We thus obtain
\begin{align*}
&F_\mu(z;\ba,|\,q)\,F_\nu(z;\ba\,|\,q)\\
&=\sum_{k\geqslant 0}
\frac{(\ba;q)_k(\ba;q)_{\mu-k}}{(q;q)_k(q;q)_{\mu-k}}
\frac{(\ba;q)_\nu}{(q;q)_\nu}z^{2k}\frac
{(\ba^2q^{\mu+\nu-2k},q^{1+\nu-k};q)_k}
{(\ba q^{1+\mu+\nu-2k},\ba q^{\nu-k};q)_k}\allowdisplaybreaks[0]\\
&\qquad\quad\times\sum_{0\le l\le k}
\frac{1-\ba q^{\mu+\nu-2k+2l}}{1-\ba q^{\mu+\nu-2k}}
\frac{(\ba q^{\mu+\nu-2k},\ba,\ba q^{\mu-k},\ba q^{\nu-k},\
ba^2q^{\mu+\nu-k},q^{-k};q)_l}
{(q,q^{1+\mu+\nu-2k},q^{1+\nu-k},q^{1+\mu-k},q^{1-k}/\ba,
\ba q^{1+\mu+\nu-k};q)_l}\allowdisplaybreaks[0]\\
&\qquad\qquad\qquad\times
\frac{(q^{1+\mu+\nu-2k};q)_{2l}}
{(\ba^2q^{\mu+\nu-2k};q)_{2l}}\Big(\frac q\ba\Big)^l\\
&=\sum_{l\geqslant 0}\sum_{k\geqslant l}
\frac{(\ba;q)_k(\ba;q)_{\mu-k}}{(q;q)_k(q;q)_{\mu-k}}
\frac{(\ba;q)_\nu}{(q;q)_\nu}z^{2k}\frac
{(\ba^2q^{\mu+\nu-2k},q^{1+\nu-k};q)_k} {(\ba q^{1+\mu+\nu-2k},\ba
q^{\nu-k};q)_k}
\frac{1-\ba q^{\mu+\nu-2k+2l}}{1-\ba q^{\mu+\nu-2k}}\allowdisplaybreaks[0]\\
&\qquad\qquad\;\times
\frac{(\ba q^{\mu+\nu-2k},\ba,\ba q^{\mu-k},\ba q^{\nu-k},\
ba^2q^{\mu+\nu-k},q^{-k};q)_l}
{(q,q^{1+\mu+\nu-2k},q^{1+\nu-k},q^{1+\mu-k},q^{1-k}/\ba,
\ba q^{1+\mu+\nu-k};q)_l}
\frac{(q^{1+\mu+\nu-2k};q)_{2l}}
{(\ba^2q^{\mu+\nu-2k};q)_{2l}}\Big(\frac q\ba\Big)^l\allowdisplaybreaks[0]\\
&\underset{k\mapsto k+l}{=} \sum_{l\geqslant 0}\sum_{k\geqslant 0}
\frac{(\ba;q)_l(\ba;q)_{\mu+\nu-2k-l}}{(q;q)_l(q;q)_{\mu+\nu-2k-l}}
z^{2l+2k}\;\frac{1-\ba q^{\mu+\nu-2k}}{1-\ba}\allowdisplaybreaks[0]\\
&\qquad\qquad\qquad\times
\frac{(q;q)_{\mu+\nu-2k}(\ba;q)_{\mu-k}(\ba;q)_{\nu-k}(\ba;q)_k
(\ba^2;q)_{\mu+\nu-k}}
{(\ba^2;q)_{\mu+\nu-2k}(q;q)_{\mu-k}(q;q)_{\nu-k}(q;q)_k(q\ba;q)_{\mu+\nu-k}}\\
&=\sum_{k\geqslant 0}
\frac{(q;q)_{\mu+\nu-2k}(\ba;q)_{\mu-k}(\ba;q)_{\nu-k}(\ba;q)_k
(\ba^2;q)_{\mu+\nu-k}}
{(\ba^2;q)_{\mu+\nu-2k}(q;q)_{\mu-k}(q;q)_{\nu-k}(q;q)_k(q\ba;q)_{\mu+\nu-k}}\,
\frac{1-\ba
q^{\mu+\nu-2k}}{1-\ba}z^{2k}\allowdisplaybreaks[0]\\&\qquad\quad\times
\sum_{l\ge
0}\frac{(\ba;q)_l(\ba;q)_{\mu+\nu-2k-l}}{(q;q)_l(q;q)_{\mu+\nu-2k-l}}
z^{2l},
\end{align*}
which, after identifying the inner sum as $F_{\mu+\nu-2k}(z;\ba\,|\,q)$
according to \eqref{defF}, establishes
the assertion.
\end{proof}

\begin{corollary}
Rogers' linearization formula for the continuous ultraspherical
polynomials in \eqref{linear} is true.
\end{corollary}
\begin{proof}
In Theorem~\ref{thm:linearF} choose $\mu=m$ and $\nu=n$ for two
nonnegative integers $m$ and $n$. The identity \eqref{linearF} then
reduces, after dividing both sides by $z^{n+m}$, to \eqref{linear}.
\end{proof}
\end{appendix}

\medskip
\textbf{Acknowledgements.} The authors would like to thank the anonymous referees for very careful readings of a previous version of this paper.
Both authors thank Prof.~Shishuo~Fu
at Chongqing University in China for hosting them for several days in
September 2018, during which parts of the present research were performed.



\begin{thebibliography}{99}

\bibitem{Andrews99}G.E. Andrews,
$q$-Analogs of the binomial coefficient congruences of
{B}abbage, {W}olstenholme and {G}laisher,
{\em Discrete Math.} \textbf{204} (1999), 15--25.

\bibitem{Andrews} G.E. Andrews,
{\em The Theory of Partitions},
Cambridge University Press, Cambridge, 1998.

\bibitem{AAR} G.E.~Andrews, R.A.~Askey, and R.~Roy,
{\em Special functions},
Encyclopedia of Mathematics and Its Applications \textbf{71},
Cambridge University Press, Cambridge, 1999.

\bibitem{AO}S. Ahlgren and K. Ono,
Gaussian hypergeometric series evaluation and {A}p\'ery number congruences,
{\em J. Reine Angew. Math.} \textbf{518} (2000), 187--212.

\bibitem{AI} R.A.~Askey and M.E.H.~Ismail,
A generalization of the ultraspherical polynomials,
in {\em Studies in Pure Mathematics}, Birkh\"auser,
Basel, 1983; pp. 55–78.

\bibitem{B} W.N.~Bailey,
{\em Generalized hypergeometric series},
Cambridge University Press, Cambridge, 1964.

\bibitem{Br}D. Bressoud,
Some identities for terminating $q$-series,
{\em Math.\ Proc.\ Camb.\ Phil.\ Soc.\ }\textbf{81} (1981), 211--223.

\bibitem{CP}H.-Q. Cao and H. Pan,
Factors of alternating binomial sums,
{\em Adv. Appl. Math.} \textbf{45} (2010), 96--107.

\bibitem{CXH}Y.-G. Chen, X.-Y. Xie, and B. He,
On some congruences of certain binomials sums,
{\em Ramanujan J.\ }\textbf{40} (2016), 237--244.

\bibitem{FH}J. F\"urlinger and J. Hofbauer,
$q$-{C}atalan numbers,
{\em J. Combin. Theory, Ser. A} \textbf{2} (1985), 248--264.

\bibitem{G}G. Gasper,
Rogers' linearization formula for the continuous
$q$-ultraspherical polynomials and quadratic transformation formulas,
{\em SIAM J.\ Math.\ Anal.\ }\textbf{16} (1985), 1061--1071.

\bibitem{Gasper}G. Gasper,
Summation, transformation, and expansion formulas for bibasic series,
{\em Trans. Amer. Math. Soc.\ }\textbf{312} (1989), 257--277.

\bibitem{GR0}G. Gasper and M. Rahman,
An indefinite bibasic summation formula and some quadratic, cubic and
quartic summation and transformation formulas,
{\em Canad. J. Math.\ }\textbf{42} (1990), 1--27.

\bibitem{GR} G.~Gasper and M.~Rahman,
{\em Basic Hypergeometric Series}, second edition,
Encyclopedia of Mathematics and Its Applications \textbf{96},
Cambridge University Press, Cambridge, 2004.

\bibitem{Guillera3}J. Guillera,
WZ pairs and $q$-analogues of {R}amanujan series for $1/\pi$,
{\em J. Difference Equ. Appl.\ }\textbf{24} (2018), 1871--1879.

\bibitem{GuZu}J. Guillera and W. Zudilin,
``Divergent" {R}amanujan-type supercongruences,
{\em Proc. Amer. Math. Soc.} \textbf{140} (2012), 765--777.

\bibitem{Guo2018}V.J.W. Guo,
A $q$-analogue of a {R}amanujan-type supercongruence involving
central binomial coefficients,
{\em J. Math. Anal. Appl.} \textbf{458} (2018), 590--600.

\bibitem{Guo-a2}V.J.W. Guo,
A $q$-analogue of the (A.2) supercongruence of Van Hamme for primes
$p\equiv 1\pmod 4$,
{\em Rev. R. Acad. Cienc. Exactas F\'is. Nat., Ser. A Mat.} \textbf{114} (2020), Art. 123.

\bibitem{Guo-jmaa}V.J.W. Guo,
$q$-Analogues of Dwork-type supercongruences,
{\em J. Math. Anal. Appl.} \textbf{487} (2020), Art.~124022.

\bibitem{Guo-rama}V.J.W. Guo,
$q$-Analogues of three Ramanujan-type formulas for $1/\pi$,
{\em Ramanujan J.} \textbf{52} (2020), 123--132.

\bibitem{Guo-mod4}V.J.W. Guo,
$q$-Supercongruences modulo the fourth power of a cyclotomic
polynomial via creative microscoping,
{\em Adv. Appl. Math.} \textbf{120} (2020), Art. 102078.

\bibitem{Guo4}V.J.W. Guo,
$q$-Analogues of two ``divergent" {R}amanujan-type supercongruences,
{\em Ramanujan J.}  \textbf{52} (2020), 605--624.

\bibitem{GS1}V.J.W. Guo and M.J. Schlosser,
Some new $q$-congruences for truncated basic hypergeometric series,
{\em Symmetry} \textbf{11}(2) (2019), Art. 268.

\bibitem{GS2}V.J.W. Guo and M.J. Schlosser,
Proof of a basic hypergeometric supercongruence modulo the
fifth power of a cyclotomic polynomial,
{\em J. Difference Equ. Appl.} \textbf{25} (2019), 921--929.

\bibitem{GS3}V.J.W. Guo and M.J. Schlosser,
Some new $q$-congruences for truncated basic hypergeometric series:
even powers, {\em Results Math.} \textbf{75} (2020), Art. 1.

\bibitem{GS4}V.J.W. Guo and M.J. Schlosser,
A family of $q$-hypergeometric congruences modulo the fourth power
of a cyclotomic polynomial,
{\em Israel J. Math.}, to appear.

\bibitem{GW}V.J.W. Guo and S.-D. Wang,
Some congruences involving fourth powers of central $q$-binomial coefficients,
{\em Proc. Roy. Soc. Edinburgh Sect.~A} \textbf{150} (2020), 1127--1138.

\bibitem{GZ15}V.J.W. Guo and J. Zeng,
Some $q$-supercongruences for truncated basic hypergeometric series,
{\em Acta Arith.}  \textbf{171} (2015), 309--326.

\bibitem{GuoZu}V.J.W. Guo and W. Zudilin,
A $q$-microscope for supercongruences,
{\em Adv. Math.} \textbf{346} (2019), 329--358.

\bibitem{GuoZu2}V.J.W. Guo and W. Zudilin,
 A common $q$-analogue of two supercongruences,
 {\em Results Math.} \textbf{75} (2020), Art. 46.

\bibitem{Hu}D.W. Hu,
On combinatorial congruences and additive combinatorics,
{\em Ph.D. thesis}, Nanjing University, China, 2017.

\bibitem{Ibook} M.E.H.~Ismail,
{\em Classical and quantum orthogonal polynomials in one variable},
Encyclopedia of Mathematics and Its Applications \textbf{98},
Cambridge University Press, Cambridge, 2005.

\bibitem{IRS} M.E.H.~Ismail, M.~Rahman, and S.~Suslov,
Some summation theorems and transformations for $q$-series,
{\em Canad.\ J.\ Math.\ }\textbf{49} (1997), 543--567.

\bibitem{JK}A. Jana and G. Kalita,
Proof of some conjectural supercongruences of {G}uo and {S}chlosser,
{\em Ramanujan J.} (2020),
\texttt{https://doi.org/10.1007/s11139-019-00221-5}.

\bibitem{Kilbourn}T. Kilbourn,
An extension of the {A}p\'ery number supercongruence,
{\em Acta Arith.} \textbf{123} (2006), 335--348.

\bibitem{LSW} R.~Langer, M.J.~Schlosser, and S.O.~Warnaar,
Theta functions, elliptic hypergeometric series, and
{K}awanaka's {M}acdonald polynomial conjecture,
SIGMA 05 (2009), 055, 20 pp.

\bibitem{LPZ}J. Liu, H. Pan and Y. Zhang,
A generalization of {M}orley's congruence,
{\em Adv. Difference Equ.} (2015), 2015:254, 7 pp.

\bibitem{Liu}J-C. Liu,
On {V}an {H}amme's (A.2) and (H.2) supercongruences,
{\em J. Math. Anal. Appl.} \textbf{471}(1-2) (2019), 613--622.

\bibitem{Liu2}J.-C. Liu,
Some supercongruences arising from symbolic summation,
{\em J. Math. Anal. Appl.} \textbf{488} (2020) Art. 124062.

\bibitem{LP}J.-C. Liu and F. Petrov,
Congruences on sums of $q$-binomial coefficients,
{\em Adv. Appl. Math.} \textbf{116} (2020), Art.~102003.

\bibitem{Long}L. Long,
Hypergeometric evaluation identities and supercongruences,
{\em Pacific J. Math.} \textbf{249} (2011), 405--418.

\bibitem{LR}L. Long and R. Ramakrishna,
Some supercongruences occurring in truncated hypergeometric series,
{\em Adv. Math.} \textbf{290} (2016), 773--808.

\bibitem{MZ}G.-S. Mao and T. Zhang,
Proof of {S}un's conjectures on super congruences and the divisibility
of certain binomial sums,
{\em Ramanujan J.\ }\textbf{50}(1) (2019), 1--11.

\bibitem{MO}D. McCarthy and R. Osburn,
A $p$-adic analogue of a formula of {R}amanujan,
{\em Arch. Math.} \textbf{91} (2008), 492--504.

\bibitem{Mortenson} E. Mortenson,
A $p$-adic supercongruence conjecture of {V}an {H}amme,
{\em Proc. Amer. Math. Soc.} \textbf{136} (2008), 4321--4328.

\bibitem{NP}H.-X. Ni and H. Pan,
Divisibility of some binomial sums,
{\em Acta Arith.} \textbf{194} (2020), 367--382.

\bibitem{R}M. Rahman,
The linearization of the product of the continuous $q$-{J}acobi polynomials,
{\em Canad.\ J.\ Math.\ }\textbf{33} (1981), 255--284.

\bibitem{Rahman}M. Rahman,
Some quadratic and cubic summation formulas for basic hypergeometric series,
{\em Canad. J. Math.} \textbf{45} (1993),  394--411.

\bibitem{R3} L.J.~Rogers,
Third memoir on the expansion of certain infinite products,
Proc.\ London Math.\ Soc.\ \textbf{26} (1894), 15--32.

\bibitem{SP}L.-L. Shi and H. Pan,
A $q$-analogue of {W}olstenholme's harmonic series congruence,
{\em Amer. Math. Monthly} \textbf{114} (2005), 529--531.

\bibitem{Straub}A. Straub,
Supercongruences for polynomial analogs of the {A}p\'ery numbers,
{\em Proc. Amer. Math. Soc.} \textbf{147} (2019), 1023--1036.

\bibitem{Sun4}Z.-W. Sun,
Super congruences and {E}uler numbers,
{\em Sci. China Math.} \textbf{54} (2011), 2509--2535.

\bibitem{Tannery}J. Tannery,
{\em Introduction a la Th\'eorie des Fonctions d'une Variable}, 2 ed., Tome 1,
Libraire Scientifique A.\ Hermann, Paris, 1904.

\bibitem{Tauraso1}R. Tauraso,
$q$-Analogs of some congruences involving {C}atalan numbers,
{\em Adv. Appl. Math.} \textbf{48} (2009), 603--614.

\bibitem{Tauraso2}R. Tauraso,
Some $q$-analogs of congruences for central binomial sums,
{\em Colloq. Math.} \textbf{133} (2013), 133--143.

\bibitem{Hamme}L. Van Hamme,
Some conjectures concerning partial sums of generalized hypergeometric series,
in: {\em $p$-Adic Functional Analysis} (Nijmegen, 1996),
Lecture Notes in Pure and Appl. Math. \textbf{192},
Dekker, New York (1997), 223--236.

\bibitem{CWang}C. Wang,
Symbolic summation methods and hypergeometric supercongruences,
{\em J. Math. Anal. Appl.} \textbf{488} (2020), Art. 124068.

\bibitem{WY}X. Wang and M. Yue,
A $q$-analogue of the (A.2) supercongruence of Van Hamme for any prime
$p\equiv 3\pmod{4}$,
{\em Int. J. Number Theory} \textbf{16} (2020), 1325--1335.

\bibitem{Zudilin2}W. Zudilin,
Congruences for $q$-binomial coefficients,
{\em Ann. Combin.}  \textbf{23} (2019), 1123--1135.


\end{thebibliography}
\end{document}